\documentclass[a4paper]{amsart}
\usepackage[utf8]{inputenc}
\usepackage{amsmath}
\usepackage{amsthm}
\usepackage{amsfonts,amssymb}
\usepackage[foot]{amsaddr}
\usepackage{mathtools}
\usepackage{mathshortcuts}
\usepackage{enumitem}
\usepackage[hypertexnames=false]{hyperref}
\usepackage{cleveref}
\usepackage{autonum}
\usepackage{float}
\usepackage{subcaption}
\usepackage{mleftright}
\usepackage{todonotes}
\usepackage[left=3cm,right=3cm,top=3.5cm,bottom=3.5cm]{geometry}
\usepackage{tikz}
\usetikzlibrary{positioning,shapes.geometric,decorations.markings,arrows}

\mleftright

\newcommand{\alphastarrho}{{}_{\alpha^*}\!\rho}

\theoremstyle{plain}
\newtheorem{thm}{Theorem}[section]

\theoremstyle{definition}
\newtheorem{defn}[thm]{Definition}

\newtheorem{rem}[thm]{Remark}

\crefname{defn}{Definition}{Definitions}
\Crefname{defn}{Definition}{Definitions}
\crefname{thm}{Theorem}{Theorems}
\Crefname{thm}{Theorem}{Theorems}
\crefname{lem}{Lemma}{Lemmas}
\Crefname{lem}{Lemma}{Lemmas}
\crefname{rem}{Remark}{Remarks}
\Crefname{rem}{Remark}{Remarks}
\crefname{prop}{Proposition}{Propositions}
\Crefname{prop}{Proposition}{Propositions}
\crefname{cor}{Corollary}{Corollaries}
\Crefname{cor}{Corollary}{Corollaries}
\crefname{section}{Section}{Sections}
\Crefname{section}{Section}{Sections}
\crefname{equation}{}{}
\Crefname{equation}{}{}
\crefname{figure}{Figure}{Figures}
\Crefname{figure}{Figure}{Figures}
\crefname{appendix}{Appendix}{Appendices}
\Crefname{appendix}{Appendix}{Appendices}

\labelcrefformat{subequation}{#2(#1)#3}
\labelcrefrangeformat{subequation}{#3(#1)#4 to #5(#2)#6}

\title{The $F$-Symbols for the $\mathcal{H}_3$ Fusion Category}
\author{Tobias J. Osborne}
\email{tobias.osborne@itp.uni-hannover.de}
\author{Deniz E. Stiegemann}
\email{deniz@stiegemann.com}
\author{Ramona Wolf}
\email{ramona.wolf@itp.uni-hannover.de}

\address{Institut für theoretische Physik, Leibniz Universität Hannover, Appelstraße 2, 30167 Hannover, Germany}

\begin{document}

\begin{abstract}
We present a solution for the $F$-symbols of the $\mathcal{H}_3$ fusion category, which is Morita equivalent to the even parts of the Haagerup subfactor. This solution has been computed by solving the pentagon equations and using several properties of trivalent categories.
\end{abstract}

\maketitle



\section{Introduction}

Subfactors \cite{jonesIntroductionSubfactors1997} play an increasingly interesting role in mathematics and physics. Particularly striking are applications in physics: here mathematical investigations into the theory of quantum fields have led to a conjectured  correspondence between subfactors and conformal field theories (CFTs) \cite{Jones2010VonNA}. This conjecture builds on original work of Doplicher \cite{doplicherNewDualityTheory1989} and was later bolstered by Bischoff \cite{bischoffRelationSubfactorsArising2015,bischoffRemarkCFTRealization2016}. A considerable body of evidence for the conjecture has now been collected (see, e.g., \cite{xuExamplesSubfactorsConformal2018,calegariCyclotomicIntegersFusion2011}). There are, however, gaps: there is a set of exceptional subfactors with no known counterpart CFT. The Haagerup subfactor \cite{haagerupPrincipal1993,AH99}, which is the smallest (finite-depth, irreducible, hyperfinite) subfactor with index more than 4, is arguably the first example that will likely require new techniques to build a corresponding CFT. Interestingly, it is possible to indirectly reason about a hypothetical Haagerup CFT, and quite a lot is already known about its potential properties \cite{evansExoticnessRealisabilityTwisted2011}.

Associated with the even parts of the Haagerup subfactor are two unitary fusion categories, here denoted $\mathcal{H}_1$ and $\mathcal{H}_2$. There is a third fusion category which is Morita equivalent to $\mathcal{H}_1$ and $\mathcal{H}_2$ but not isomorphic to either of them \cite{GS12}. This fusion category, which is denoted $\mathcal{H}_3$, has six simple objects and is of particular interest for this paper as it admits a simple skein theory. Owing to its comparative simplicity of formulation, the resulting \emph{trivalent category} \cite{MPS17} is especially attractive as a basis for building CFT models. 

An important quantity for any unitary fusion category, required for any deeper investigation of physical models, are its $F$- or $6j$-symbols, determined by solving its \emph{pentagon equations}. The $F$-symbols are notoriously difficult to obtain in general and explicit solutions are known only for a handful of cases (see, e.g., \cite{Bonderson} for list of many explicit solutions). Here the $\mathcal{H}_3$ fusion category presents no shortage of challenges: after eliminating trivial equations there are 41391 equations determining 1431 unknowns, a task at the limit of what is easily computable with current state-of-the-art technology and algorithms. 

It is worth emphasising that the $F$-symbols for $\mathcal{H}_3$ have actually already, in principle, been determined by Izumi \cite{Izumi}, who exploited Cuntz-algebra techniques to describe them. Our goal in this paper, however, is to develop methods admitting an easy generalisation to other unitary fusion categories. In particular, we exploit skein-theoretic information to provide ``seed'' data for the subsequent application of standard equation solvers. This approach promises to generalise to a variety of interesting cases, possibly including the extended Haagerup subfactors \cite{EHplanaralg,EHfusioncats}. 

The outline of this paper is as follows. In \cref{sec:2_fusioncat} we review the definition of the $\mathcal{H}_3$ fusion category. This is followed, in \cref{sec:3_result}, by a brief statement of our main result. The trivalent category methods we exploited are described in \cref{sec:5_methods}. The simplifications allowed by fixing a gauge are described in \cref{sec:6_gauge} (we work in a gauge where the solutions are real). Finally, we conclude in \cref{sec:7_conclusion}. The diagrammatic calculus of trivalent categories is reviewed in \cref{app:TriCats} and the explicit solution for the $F$-symbols written out in \cref{app:Fsymbols}. 

While completing this work we learnt that Matthew Titsworth \cite{Priv} has also previously obtained the solution to the pentagon equations for $\mathcal{H}_3$, and further determined the monoidal classes of the solutions according to the methods of \cite{haggeGeometricInvariantsFusion2015}.


\section{The $\mathcal{H}_3$ fusion category}
	\label{sec:2_fusioncat}
	The category we study in this work is a fusion category, which is defined as follows:
		\begin{defn}
			A \textbf{fusion category} over $\mathbb{C}$ is a $\mathbb{C}$-linear rigid semisimple monoidal category with finitely many simple objects (up to isomorphism) and finite-dimensional morphism spaces, such that the identity object is simple.
		\end{defn}
	For more details on fusion category, see for example \cite{ENO05}.
	The category we are considering here is the $\mathcal{H}_3$ fusion category that was constructed in~\cite{GS12}, which is Morita equivalent to the even parts of the Haagerup subfactor, which was introduced in \cite{AH99}. Its set of simple objects is given by
		\begin{equation}
			\mathcal{M}_{\mathcal{H}_3}=\{1,\alpha,\alpha^*,\rho,{}_\alpha\rho,\alphastarrho\}
		\end{equation}
	with the dimensions $d_1=d_\alpha=d_{\alpha^*}=1$, $d_\rho=d_{{}_\alpha\rho}=d_{\alphastarrho}=\frac{3+\sqrt{13}}{2}$.
	The fusion rules for this category are listed in \cref{fig:fusiontable}.
	
	\begin{figure}[H]
		\centering
		\begin{tabular}{c||c|c|c|c|c|c}
			& $1$ & $\alpha$ & $\alpha^*$ & $\rho$ & ${}_\alpha
			\rho$ & $\alphastarrho$ \\ \hline \hline
			$1$ & $1$ & $\alpha$ & $\alpha^*$ & $\rho$ & ${}_\alpha \rho$ & $\alphastarrho$ \\ \hline
			$\alpha$ & $\alpha$ & $\alpha^*$ & $1$ & ${}_\alpha \rho$ & $\alphastarrho$ & $\rho$ \\ \hline
			$\alpha^*$ & $\alpha^*$ & $1$ & $\alpha$ & $\alphastarrho$ & $\rho$ & ${}_\alpha \rho$ \\ \hline
			$\rho$ & $\rho$ & $\alphastarrho$ & ${}_\alpha \rho$ & $1+ \rho+ {}_\alpha\rho+\alphastarrho$ & $\alpha^*+ \rho+ {}_\alpha\rho+ \alphastarrho$ & $\alpha+ \rho+ {}_\alpha\rho+\alphastarrho$ \\ \hline
			${}_\alpha\rho$ & ${}_\alpha\rho$ & $\rho$ & $\alphastarrho$ & $\alpha+ \rho+ {}_\alpha\rho+ \alphastarrho$ & $1+ \rho+ {}_\alpha\rho+ \alphastarrho$ & $\alpha^*+ \rho+ {}_\alpha\rho+ \alphastarrho$ \\ \hline
			$\alphastarrho$ & $\alphastarrho$ & ${}_\alpha\rho$ & $\rho$ & $\alpha^*+ \rho+ {}_\alpha\rho+ \alphastarrho$ & $\alpha+ \rho+ {}_\alpha\rho+ \alphastarrho$ & $1+ \rho+ {}_\alpha\rho+ \alphastarrho$
		\end{tabular}
		\caption{Fusion rules for the $\mathcal{H}_3$ fusion category.}
		\label{fig:fusiontable}
	\end{figure}
	Throughout the paper, we will denote the vector space that corresponds to the fusion of two objects $a$ and $b$ to an object $c$, i.e.\ the morphism space $\hom(c,a\otimes b)$ by $V_c^{ab}$ with $\dim V_c^{ab}=N_c^{ab}$.
	Additionally, we know that there is a trivalent category that corresponds to $\mathcal{H}_3$ (see~\cite{MPS17}). It is the full subcategory of $\mathcal{H}_3$ whose objects are $\rho$ and its tensor powers. The Karoubi envelope (i.e., the idempotent completion) of this trivalent category is the original fusion category $\mathcal{H}_3$.
	
	In the following, we will need some properties of monoidal categories. A \emph{pivotal} category $\mathcal{C}$ is a rigid monoidal category (i.e.\ every object has a dual that fulfils certain conditions) such that $x^{**}=x$ for every $x\in\mathrm{obj}(\mathcal{C})$. A category is \emph{evaluable} if $\dim\hom(1,1)=1$ and it is said to be \emph{non-degenerate} if for every morphism $f\in\hom(a,b)$ there is a morphism $f'\in\hom(b,a)$ such that $\mathrm{tr}(f\circ f')\neq 0$ in $\hom(1,1)$. For a pivotal category $\mathcal{C}$ and a fixed object $X$, $\mathcal{C}_n$ denotes the morphism space $\hom(1,X^{\otimes n})$. 
		\begin{defn}
			A \textbf{trivalent category} $(\mathcal{C}, X, \tau)$ is a non-degenerate evaluable pivotal category over $\mathbb{C}$ with an object $X$ with $\dim\mathcal{C}_1=0$, $\dim\mathcal{C}_2=1$, and $\dim\mathcal{C}_3=1$, with a rotationally invariant morphism $\tau\in\mathcal{C}_3$ called the trivalent vertex, such that the category is generated (as a pivotal category) by $\tau$.
		\end{defn}
	
		\begin{rem}
			The object $X$ is symmetrically self-dual (see \cite{MPS17}, Lemma 2.2), which implies that we can drop the orientations on string diagrams. In summary, every unoriented planar trivalent graph with $n$ boundary points can be interpreted as an element of $\mathcal{C}_n$. This provides some useful simplifications in the diagrammatic calculus of fusion categories (see \cref{app:A_DiagramCalc} for more details), which we exploit in \cref{sec_additionaleqs,sec:MethodsTriv}.
		\end{rem}
	
		\begin{rem}
			For $\mathcal{H}_3$, the chosen object $X$ is $\rho$, since it generates all the other simple objects in the category.
		\end{rem}


\section{Result}
\label{sec:3_result}

The main result of this paper is summarised in the following theorem:

\begin{thm}
	The pentagon equations can be solved by the $F$-symbols given in \cref{app:Fsymbols}, which is a real solution with two parameters $p_1,p_2\in\{-1,+1\}$.
\end{thm}

All occurring values are in the interval $[-1,+1]$, so we can visualise this solution as follows: We represent the value $+1$ with a black pixel and the value $-1$ with a white pixel. The values in between are represented by a green pixel whose darkness depends on where the value lies in the interval $[-1,+1]$, e.g.\ a value close to $+1$ is represented by a very dark green pixel, a value close to $-1$ by a very light green pixel. For $p_1=p_2=+1$, this is shown in \cref{fig:pixels}, where the order of the $F$-symbols is chosen randomly.

\begin{figure}[H]
	\includegraphics[width=0.65\textwidth]{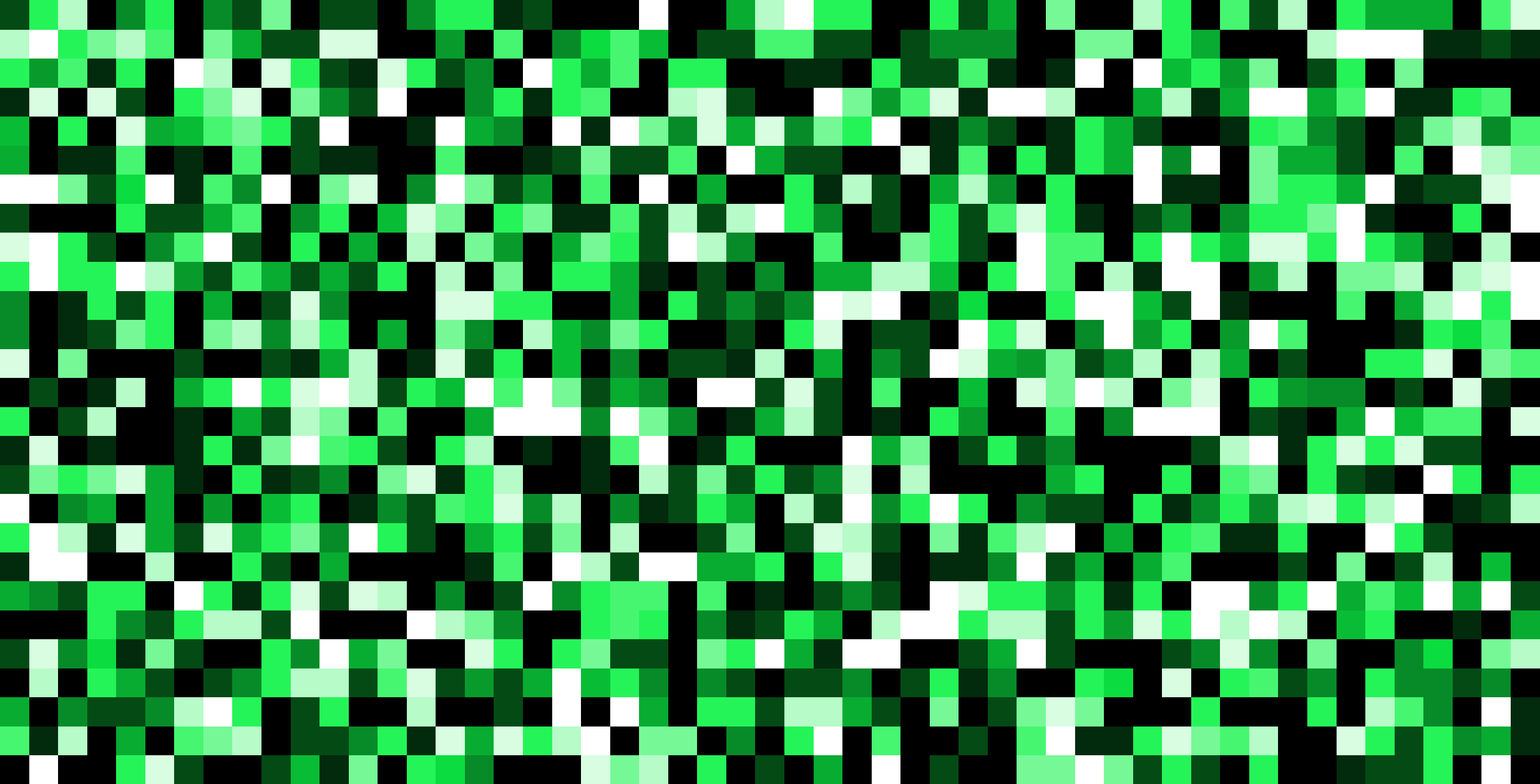}
	\caption{A visualisation of the $F$-symbols for $p_1=p_2=+1$.}
	\label{fig:pixels}
\end{figure}

\section{The pentagon equations}		

In this section, we explain the \emph{pentagon equations}, which are the fundamental equations we have to solve to obtain a solution for the $F$-symbols.
$F$-symbols are unitary isomorphisms between different parenthesis structures of the tensor product of multiple objects, i.e.\ they are maps
	\begin{equation}
		F_u^{xyz}: \hom(u,(x\otimes y)\otimes z)\rightarrow \hom(u,x\otimes (y\otimes z)).
	\end{equation}
In terms of fusion diagrams, this is depicted as
	\begin{equation}
		\begin{tikzpicture}[baseline={([yshift=-.5ex]current bounding box.center)},scale=0.4]
			\node at (-5,4.6) {$u$};
			\node at (-7,8.4) {$x$};
			\node at (-5,8.4) {$y$};
			\node at (-3,8.4) {$z$};
			\draw (-5,5) -- (-5,6);
			\draw (-5,6) -- (-6,7) node[below] {$k\ $};
			\draw (-5,6) -- (-4,7);
			\draw (-6,7) -- (-7,8);
			\draw (-6,7) -- (-5,8);
			\draw (-4,7) -- (-3,8);
		\end{tikzpicture}
		=\sum_j \left(F_u^{xyz}\right)_{jk}
		\begin{tikzpicture}[baseline={([yshift=-.5ex]current bounding box.center)},scale=0.4]
			\node at (5,4.6) {$u$};
			\node at (7,8.4) {$z$};
			\node at (5,8.4) {$y$};
			\node at (3,8.4) {$x$};
			\draw (5,5) -- (5,6);
			\draw (5,6) -- (6,7) node[below] {$\ j$};
			\draw (5,6) -- (4,7);
			\draw (6,7) -- (7,8);
			\draw (6,7) -- (5,8);
			\draw (4,7) -- (3,8);
		\end{tikzpicture}
	\end{equation}
They have to fulfil a set of conditions called the \emph{pentagon equations}, which means that the diagram in \cref{fig:pentagon} has to commute.
	\begin{figure}[H]
		\begin{tikzpicture}[scale=0.65];	
			\node (u) at (-1,-0.25) {$u$};
			\node (x) at (-2.5,2.2) {$x$};
			\node (y) at (-1.5,2.2) {$y$};
			\node (z) at (-0.5,2.2) {$z$};
			\node (w) at (0.5,2.2) {$w$};
			\node (a) at (-1.9,1.1) {$a$};
			\node (b) at (-1.4,0.6) {$b$};
			\draw (u) -- (-1,0.5);
			\draw (-1,0.5) -- (-1.5,1);
			\draw (-1.5,1) -- (-2,1.5);
			\draw (-2,1.5) -- (-2.5,2);
			\draw (-1,0.5) -- (-0.5,1);
			\draw (-0.5,1) -- (0,1.5);
			\draw (0,1.5) -- (0.5,2);
			\draw (-2,1.5) -- (-1.5,2);
			\draw (-1.5,1) -- (-1,1.5);
			\draw (-1,1.5) --(-0.5,2);
			\draw[->] (0.2,2.9) to node[above] {$F_u^{azw}$} (2.9,4.1);
			\node (u2) at (5,2.75) {$u$};
			\node (x2) at (3.5,5.2) {$x$};
			\node (y2) at (4.5,5.2) {$y$};
			\node (z2) at (5.5,5.2) {$z$};
			\node (w2) at (6.5,5.2) {$w$};
			\node (a2) at (4.1,4.1) {$a$};
			\node (c2) at (5.9,4.1) {$c$};
			\draw (u2) -- (5,3.5);
			\draw (5,3.5) -- (4.5,4);
			\draw (4.5,4) -- (4,4.5);
			\draw (4,4.5) -- (3.5,5);
			\draw (5,3.5) -- (5.5,4);
			\draw (5.5,4) -- (6,4.5);
			\draw (6,4.5) -- (6.5,5);
			\draw (4,4.5) -- (4.5,5);
			\draw (6,4.5) -- (5.5,5);
			\draw[->] (7.1,4.1) to node[above]{$F_u^{xyc}$} (9.8,2.9);
			\node (u3) at (11,-0.25) {$u$};
			\node (x3) at (9.5,2.2) {$x$};
			\node (y3) at (10.5,2.2) {$y$};
			\node (z3) at (11.5,2.2) {$z$};
			\node (w3) at (12.5,2.2) {$w$};
			\node (d3) at (11.4,0.6) {$d$};
			\node (c3) at (11.9,1.1) {$c$};
			\draw (u3) -- (11,0.5);
			\draw (11,0.5) -- (10.5,1);
			\draw (10.5,1) -- (10,1.5);
			\draw (10,1.5) -- (9.5,2);
			\draw (11,0.5) -- (11.5,1);
			\draw (11.5,1) -- (12,1.5);
			\draw (12,1.5) -- (12.5,2);
			\draw (11.5,1) -- (10.5,2);
			\draw (12,1.5) -- (11.5,2);
			\draw[->] (-1,-0.9) to node[above] {$\ \ \ \ F_b^{xyz}$} (0,-2.1);
			\node (u4) at (1.5,-5.25) {$u$};
			\node (x4) at (0,-2.8) {$x$};
			\node (y4) at (1,-2.8) {$y$};
			\node (z4) at (2,-2.8) {$z$};
			\node (w4) at (3,-2.8) {$w$};
			\node (e4) at (1.4,-3.9) {$e$};
			\node (b4) at (1.1,-4.4) {$b$};
			\draw (u4) -- (1.5,-4.5);
			\draw (1.5,-4.5) -- (1,-4);
			\draw (1,-4) -- (0.5,-3.5);
			\draw (0.5,-3.5) -- (0,-3);
			\draw (1.5,-4.5) -- (2,-4);
			\draw (2,-4) -- (2.5,-3.5);
			\draw (2.5,-3.5) -- (3,-3);
			\draw (1.5,-3.5) -- (1,-3);
			\draw (1,-4) -- (1.5,-3.5);
			\draw (1.5,-3.5) --(2,-3);
			\draw[->] (3.5,-4) to node[above] {$F_u^{xew}$} (6.5,-4);
			\node (u5) at (8.5,-5.25) {$u$};
			\node (x5) at (7,-2.8) {$x$};
			\node (y5) at (8,-2.8) {$y$};
			\node (z5) at (9,-2.8) {$z$};
			\node (w5) at (10,-2.8) {$w$};
			\node (e5) at (8.6,-3.9) {$e$};
			\node (d5) at (8.9,-4.4) {$d$};
			\draw (u5) -- (8.5,-4.5);
			\draw (8.5,-4.5) -- (8,-4);
			\draw (8,-4) -- (7.5,-3.5);
			\draw (7.5,-3.5) -- (7,-3);
			\draw (8.5,-4.5) -- (9,-4);
			\draw (9,-4) -- (9.5,-3.5);
			\draw (9.5,-3.5) -- (10,-3);
			\draw (8.5,-3.5) -- (8,-3);
			\draw (9,-4) -- (8.5,-3.5);
			\draw (8.5,-3.5) --(9,-3);
			\draw[->] (10,-2.1) to node[above] {$F_d^{yzw}\ \ \ $} (11,-0.9);
		\end{tikzpicture}
		\caption{The pentagon equations.}
		\label{fig:pentagon}
	\end{figure}
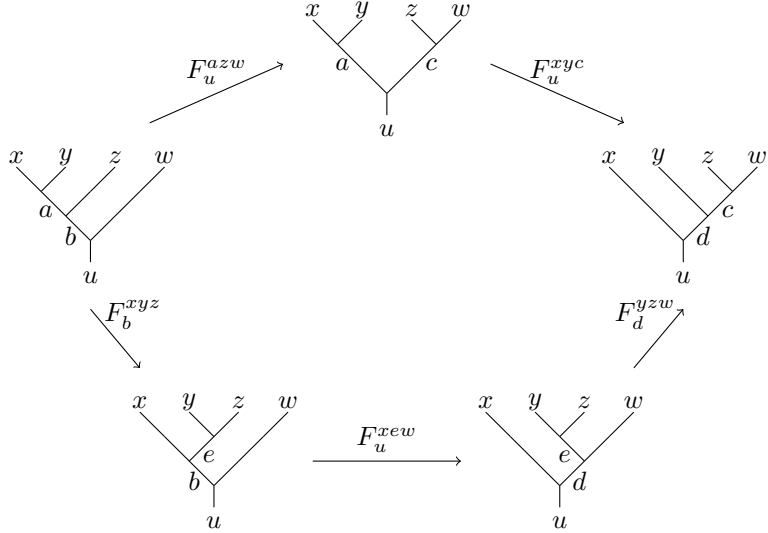
\noindent
In terms of matrix elements, this condition is expressed as
	\begin{equation}
		\label{eq:pentagon}
		\left(F_u^{xyc}\right)_{da}\left(F_u^{azw}\right)_{cb}=\sum_e\left(F_d^{yzw}\right)_{ce}\left(F_u^{xew}\right)_{db}\left(F_b^{xyz}\right)_{ea}.
	\end{equation}
Hence, to obtain the $F$-symbols for a given fusion category, we have to solve \cref{eq:pentagon}.

\section{Methods}
\label{sec:5_methods}

Beside the pentagon equations, there are some additional equations and techniques that we can use to make finding the solution easier. Some of them use the fact that the category is trivalent.	

\subsection{The triangle equation}

Since a fusion category $\mathcal{C}$ is a tensor category, we have a left and a right unit constraint: For any object $x\in\mathcal{C}$, there exist maps
	\begin{align}
		l_x&:\hom(x,1\otimes x)\to \hom(x,x)\\
		r_x&:\hom(x,x\otimes 1)\to \hom(x,x).
	\end{align}
\noindent
To ensure that these maps are compatible with the $F$-symbols, the \emph{triangle equation} has to be fulfilled, which means that the diagram in \cref{fig:triangle_eq} commutes.

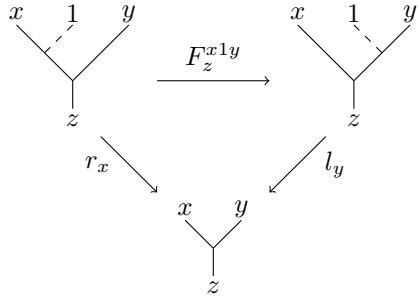
\begin{figure}[H]
	\centering
	\begin{tikzpicture}[scale=0.37]
		\node at (0,-1.4) {$z$};
		\node at (-1,1.4) {$x$};
		\node at (1,1.4) {$y$};
		\draw (0,-1) -- (0,0);
		\draw (0,0) -- (-1,1);
		\draw (0,0) -- (1,1);
		\node at (-5,4.6) {$z$};
		\node at (-7,8.4) {$x$};
		\node at (-5,8.4) {$1$};
		\node at (-3,8.4) {$y$};
		\draw (-5,5) -- (-5,6);
		\draw (-5,6) -- (-6,7);
		\draw (-5,6) -- (-4,7);
		\draw (-6,7) -- (-7,8);
		\draw[dashed] (-6,7) -- (-5,8);
		\draw (-4,7) -- (-3,8);
		\node at (5,4.6) {$z$};
		\node at (7,8.4) {$y$};
		\node at (5,8.4) {$1$};
		\node at (3,8.4) {$x$};
		\draw (5,5) -- (5,6);
		\draw (5,6) -- (6,7);
		\draw (5,6) -- (4,7);
		\draw (6,7) -- (7,8);
		\draw[dashed] (6,7) -- (5,8);
		\draw (4,7) -- (3,8);
		\draw[->] (-2,6) -- node[above] {$F_z^{x1y}$} (2,6);
		\draw[->] (-4,4) -- node[left] {$r_x\ $} (-2,2);
		\draw[->] (4,4) -- node[right] {$\ \ l_y$} (2,2);
	\end{tikzpicture}
	\caption{Diagrammatic depiction of the triangle equation.}
	\label{fig:triangle_eq}
\end{figure}

\noindent
There are some modifications of the triangle equation, which also have to be fulfilled:
\begin{figure}[H]
	\centering
	\begin{subfigure}[b]{0.45\textwidth}
		\centering
		\begin{tikzpicture}[scale=0.37]
			\node at (0,-1.4) {$z$};
			\node at (-1,1.4) {$x$};
			\node at (1,1.4) {$y$};
			\draw (0,-1) -- (0,0);
			\draw (0,0) -- (-1,1);
			\draw (0,0) -- (1,1);
			\node at (-5,4.6) {$z$};
			\node at (-7,8.4) {$1$};
			\node at (-5,8.4) {$x$};
			\node at (-3,8.4) {$y$};
			\draw (-5,5) -- (-5,6);
			\draw (-5,6) -- (-6,7);
			\draw (-5,6) -- (-4,7);
			\draw[dashed] (-6,7) -- (-7,8);
			\draw (-6,7) -- (-5,8);
			\draw (-4,7) -- (-3,8);
			\node at (5,4.6) {$z$};
			\node at (7,8.4) {$y$};
			\node at (5,8.4) {$x$};
			\node at (3,8.4) {$1$};
			\draw (5,5) -- (5,6);
			\draw (5,6) -- (6,7);
			\draw[dashed] (5,6) -- (4,7);
			\draw (6,7) -- (7,8);
			\draw (6,7) -- (5,8);
			\draw[dashed] (4,7) -- (3,8);
			\draw[->] (-2,6) -- node[above] {$F_z^{1xy}$} (2,6);
			\draw[->] (-4,4) -- node[left] {$l_x\ $} (-2,2);
			\draw[->] (4,4) -- node[right] {$\ \ l_z$} (2,2);
		\end{tikzpicture}
		\caption{Case 1.}
		\label{subfig:Trianglea}
	\end{subfigure}
	\hfill
	\begin{subfigure}[b]{0.45\textwidth}
		\centering
		\begin{tikzpicture}[scale=0.37]
			\node at (0,-1.4) {$z$};
			\node at (-1,1.4) {$x$};
			\node at (1,1.4) {$y$};
			\draw (0,-1) -- (0,0);
			\draw (0,0) -- (-1,1);
			\draw (0,0) -- (1,1);
			\node at (-5,4.6) {$z$};
			\node at (-7,8.4) {$x$};
			\node at (-5,8.4) {$y$};
			\node at (-3,8.4) {$1$};
			\draw (-5,5) -- (-5,6);
			\draw (-5,6) -- (-6,7);
			\draw[dashed] (-5,6) -- (-4,7);
			\draw (-6,7) -- (-7,8);
			\draw (-6,7) -- (-5,8);
			\draw[dashed] (-4,7) -- (-3,8);
			\node at (5,4.6) {$z$};
			\node at (7,8.4) {$1$};
			\node at (5,8.4) {$y$};
			\node at (3,8.4) {$x$};
			\draw (5,5) -- (5,6);
			\draw (5,6) -- (6,7);
			\draw (5,6) -- (4,7);
			\draw[dashed] (6,7) -- (7,8);
			\draw (6,7) -- (5,8);
			\draw (4,7) -- (3,8);
			\draw[->] (-2,6) -- node[above] {$F_z^{xy1}$} (2,6);
			\draw[->] (-4,4) -- node[left] {$r_z\ $} (-2,2);
			\draw[->] (4,4) -- node[right] {$\ \ r_y$} (2,2);
		\end{tikzpicture}
		\caption{Case 2.}
		\label{subfig:Triangleb}
	\end{subfigure}
	\caption{Variations of the triangle equation.}
	\label{fig:TriangleCases}
\end{figure}
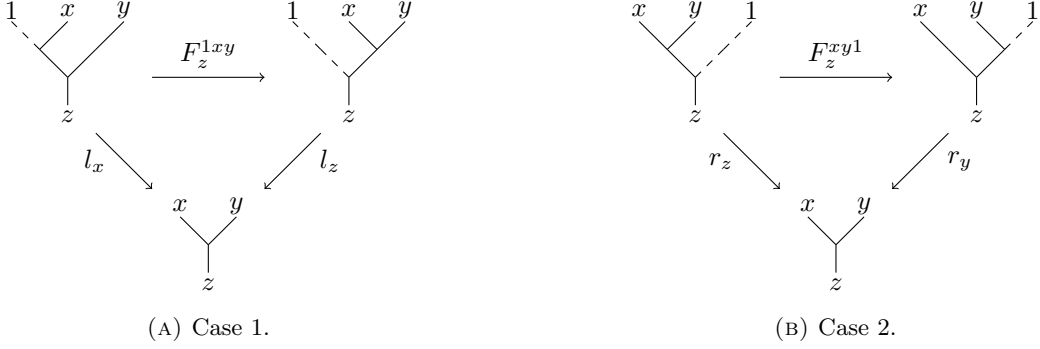

\noindent
This leads to the following observation:
\begin{thm}
	\label{thm:triangle}
	In a tensor category $\mathcal{C}$, the following equations hold whenever the corresponding vector spaces are one-dimensional:
	\begin{align}
	\begin{split}
	\label{eq:triangle}
	F^{1xy}_z F^{1zy}_x&=1,\\
	F^{xy1}_z F^{xz1}_y&=1.
	\end{split}
	\end{align}
\end{thm}
\begin{proof}
	Consider the diagram in \cref{subfig:Trianglea}. Since all the occurring maps are one-dimensional, we can treat them as complex numbers. Hence, the diagram yields the equation
		\begin{equation}
			l_x=F_z^{1xy}l_z.
			\label{Eq:Inverses1}
		\end{equation}
	Now consider a modified version of this diagram, namely the one where $x$ and $z$ are interchanged. This diagram leads to 
		\begin{equation}
			l_z=F_x^{1zy}l_x.
			\label{Eq:Inverses2}
		\end{equation}
	Inserting \cref{Eq:Inverses2} into \cref{Eq:Inverses1} yields
		\begin{align}
			l_x&=F_z^{1xy}F_x^{1zy}l_x\\
			1&=F_z^{1xy}F_x^{1zy}.
		\end{align}
	For the diagram in \cref{subfig:Triangleb} this is done analogously.
\end{proof}

\subsection{Additional equations}
\label{sec_additionaleqs}

	\begin{thm}
		\label{thm:addeq1}
		If $V_u^{abc}$ is a vector space with $\dim(V_u^{abc})\neq 0$ and $N_u^{x_3x_4}=N_{x_3}^{ax_1}=N_{x_4}^{x_2c}=N_{x_1x_2}^b=1$, the following equality holds for $x_1,x_2,x_3,x_4,x_4',y\in\{1,\alpha,\alpha^*,\rho,{}_\alpha\rho,\alphastarrho\}$:
			\begin{equation}
				\label{eq:addeq}
				 \sum_{x_3'}\big(F_u^{ax_1 x_4}\big)_{x_3' x_3}\ \big(F_{x_3'}^{x_1 x_2 c}\big)^*_{b x_4}\ \big(F_u^{abc}\big)^*_{y x_3'}=\big(F_u^{x_3 x_2 c}\big)^*_{y x_4}\ \big(F_{y}^{a x_1 x_2}\big)_{b x_3}.
			\end{equation}
	\end{thm}
	\begin{proof}
		Consider the decomposition
			\begin{equation}
				V_u^{abc}=\bigoplus_{x_1,x_2,x_3,x_4}V_u^{x_3 x_4}\otimes V_{x_3}^{a x_1}\otimes V_{x_4}^{x_2 c}\otimes V_{x_1 x_2}^{b}.
			\end{equation}
		First, start by applying the $F$-matrix act on $\bigoplus_{x_3}V_u^{x_3 x_4}\otimes V_{x_3}^{a x_1}$:
		{\allowdisplaybreaks
			\begin{align}
				\begin{tikzpicture}[scale=.5,baseline=(current bounding box.center)]
					\draw (0,0) -- (0,1);
					\draw (0,1) -- (-1,2);
					\draw (0,1) -- (1,2);
					\draw (-1,2) -- (-1,3);
					\draw (1,2) -- (1,3);
					\draw (-1,3) -- (-2,4);
					\draw (-1,3) -- (0,4);
					\draw (1,3) -- (0,4);
					\draw (1,3) -- (2,4);
					\draw (-2,4) -- (-2,6);
					\draw (0,4) -- (0,6);
					\draw (2,4) -- (2,6);
					\node (u) at (0,-0.5) {$u$};
					\node (a) at (-2,6.5) {$a$};
					\node (b) at (0,6.5) {$b$};
					\node (c) at (2,6.5) {$c$};
					\node (x3) at (-1.5,1.75) {$x_3$};
					\node (x4) at (1.5,1.75) {$x_4$};
					\node (x1) at (-0.8,3.75) {$x_1$};
					\node (x2) at (0.8,3.75) {$x_2$};
				\end{tikzpicture}
				&=\sum_{x_3'} \big(F_u^{ax_1 x_4}\big)_{x_3' x_3}\ 
				\begin{tikzpicture}[scale=.5,baseline=(current bounding box.center)]
					\draw (0,0) -- (0,1);
					\draw (0,1) -- (-1,2);
					\draw (0,1) -- (1,2);
					\draw (1,2) -- (0,3);
					\draw (1,2) -- (2,3);
					\draw (0,3) -- (0,4);
					\draw (2,3) -- (2,4);
					\draw (0,4) -- (1,5);
					\draw (2,4) -- (1,5);
					\draw (2,4) -- (3,5);
					\draw (1,5) -- (1,6);
					\draw (3,5) -- (3,6);
					\draw (-1,2) -- (-1,6);
					\node (u) at (0,-0.5) {$u$};
					\node (a) at (-1,6.5) {$a$};
					\node (b) at (1,6.5) {$b$};
					\node (c) at (3,6.5) {$c$};
					\node (x3p) at (1,1.25) {$x_3'$};
					\node (x1) at (-0.4,3) {$x_1$};
					\node (x4) at (2.5,3) {$x_4$};
					\node (x2) at (1.25,4.25) {$x_2$};
				\end{tikzpicture}\\
				&=\sum_{x_3' x_4'}\big(F_u^{ax_1 x_4}\big)_{x_3' x_3}\ \big(F_{x_3'}^{x_1 x_2 c}\big)^*_{x_4' x_4}\ 
				\begin{tikzpicture}[scale=.5,baseline=(current bounding box.center)]
					\draw (0,0) -- (0,1);
					\draw (0,1) -- (-2,3);
					\draw (0,1) -- (1,2);
					\draw (1,2) -- (0,3);
					\draw (1,2) -- (2,3);
					\draw (0,3) -- (-1,4);
					\draw (0,3) -- (1,4);
					\draw (-1,4) -- (-1,5);
					\draw (1,4) -- (1,5);
					\draw (-1,5) -- (0,6);
					\draw (1,5) -- (0,6);
					\draw (-2,3) -- (-2,7);
					\draw (0,6) -- (0,7);
					\draw (2,3) -- (2,7);
					\node (u) at (0,-0.5) {$u$};
					\node (a) at (-2,7.5) {$a$};
					\node (b) at (0,7.5) {$b$};
					\node (c) at (2,7.5) {$c$};
					\node (x4p) at (0.1,2.25) {$x_4'$};
					\node (x1) at (-1.4,4) {$x_1$};
					\node (x2) at (1.5,4) {$x_2$};
					\node (x3p) at (1,1.25) {$x_3'$};
				\end{tikzpicture}\\
				&=\sum_{x_3' x_4'}\big(F_u^{ax_1 x_4}\big)_{x_3' x_3}\ \big(F_{x_3'}^{x_1 x_2 c}\big)^*_{x_4' x_4}\ \sqrt{\frac{d_{x_1}d_{x_2}}{d_b}}\ \delta_{b,x_4'}\ 
				\begin{tikzpicture}[scale=.5,baseline=(current bounding box.center)]
					\node (u) at (6.5,-0.5) {$u$};
					\node (a) at (4.5,3.5) {$a$};
					\node (b) at (6.5,3.5) {$b$};
					\node (c) at (8.5,3.5) {$c$};
					\draw (u) -- (6.5,1);
					\draw (6.5,1) -- (7.5,2) node[pos=0.4,right] {$x_3'$};
					\draw (7.5,2) -- (6.5,3);
					\draw (7.5,2) -- (8.5,3);
					\draw (6.5,1) -- (4.5,3);
				\end{tikzpicture}\\
				&=\sum_{x_3'}\big(F_u^{ax_1 x_4}\big)_{x_3' x_3}\ \big(F_{x_3'}^{x_1 x_2 c}\big)^*_{b x_4}\ \sqrt{\frac{d_{x_1}d_{x_2}}{d_b}}\ 
				\begin{tikzpicture}[scale=.5,baseline=(current bounding box.center)]
					\node (u) at (6.5,-0.5) {$u$};
					\node (a) at (4.5,3.5) {$a$};
					\node (b) at (6.5,3.5) {$b$};
					\node (c) at (8.5,3.5) {$c$};
					\draw (u) -- (6.5,1);
					\draw (6.5,1) -- (7.5,2) node[pos=0.4,right] {$x_3'$};
					\draw (7.5,2) -- (6.5,3);
					\draw (7.5,2) -- (8.5,3);
					\draw (6.5,1) -- (4.5,3);
				\end{tikzpicture}\\
				&=\sum_{x_3',x_3''}\big(F_u^{ax_1 x_4}\big)_{x_3' x_3}\ \big(F_{x_3'}^{x_1 x_2 c}\big)^*_{b x_4}\ \sqrt{\frac{d_{x_1}d_{x_2}}{d_b}}\ \big(F_u^{abc}\big)^*_{x_3'' x_3'}\ 
				\begin{tikzpicture}[scale=.5,baseline=(current bounding box.center)]
					\node (u) at (6.5,-0.5) {$u$};
					\node (a) at (4.5,3.5) {$a$};
					\node (b) at (6.5,3.5) {$b$};
					\node (c) at (8.5,3.5) {$c$};
					\draw (u) -- (6.5,1);
					\draw (6.5,1) -- (8.5,3);
					\draw (5.5,2) -- (6.5,3);
					\draw (5.5,2) -- (4.5,3);
					\draw (6.5,1) -- (5.5,2) node[pos=0.4,left] {$x_3''\ $};
				\end{tikzpicture}.
			\end{align}
		}
		In the second step, the inverse $F$-matrix acts on the space $\bigoplus_{x_4}V_{x_3'}^{x_1 x_4}\otimes V_{x_4}^{x_2 c}$ and in the third step, we have used the rules for evaluating bigons. Alternatively, we can first let the inverse $F$-matrix act on $\bigoplus_{x_4}V_{u}^{x_3 x_4}\otimes V_{x_4}^{x_2 c}$ and the $F$-matrix on $\bigoplus_{x_3}V_{x_4'}^{x_3 x_2}\otimes V_{x_3}^{a x_1}$ afterwards:
		{\allowdisplaybreaks
			\begin{align}
				\begin{tikzpicture}[scale=.5,baseline=(current bounding box.center)]
					\draw (0,0) -- (0,1);
					\draw (0,1) -- (-1,2);
					\draw (0,1) -- (1,2);
					\draw (-1,2) -- (-1,3);
					\draw (1,2) -- (1,3);
					\draw (-1,3) -- (-2,4);
					\draw (-1,3) -- (0,4);
					\draw (1,3) -- (0,4);
					\draw (1,3) -- (2,4);
					\draw (-2,4) -- (-2,6);
					\draw (0,4) -- (0,6);
					\draw (2,4) -- (2,6);
					\node (u) at (0,-0.5) {$u$};
					\node (a) at (-2,6.5) {$a$};
					\node (b) at (0,6.5) {$b$};
					\node (c) at (2,6.5) {$c$};
					\node (x3) at (-1.5,1.75) {$x_3$};
					\node (x4) at (1.5,1.75) {$x_4$};
					\node (x1) at (-0.8,3.75) {$x_1$};
					\node (x2) at (0.8,3.75) {$x_2$};
				\end{tikzpicture}	
				&=\sum_{x_4'} \big(F_u^{x_3 x_2 c}\big)^*_{x_4' x_4}\ 
				\begin{tikzpicture}[scale=.5,baseline=(current bounding box.center)]
					\draw (0,0) -- (0,1);
					\draw (0,1) -- (-1,2);
					\draw (0,1) -- (1,2);
					\draw (-1,2) -- (-2,3);
					\draw (-1,2) -- (0,3);
					\draw (-2,3) -- (-2,4);
					\draw (0,3) -- (0,4);
					\draw (-2,4) -- (-3,5);
					\draw (-2,4) -- (-1,5);
					\draw (0,4) -- (-1,5);
					\draw (-3,5) -- (-3,6);
					\draw (-1,5) -- (-1,6);
					\draw (1,2) -- (1,6);
					\node (u) at (0,-0.5) {$u$};
					\node (a) at (-3,6.5) {$a$};
					\node (b) at (-1,6.5) {$b$};
					\node (c) at (1,6.5) {$c$};
					\node (x4p) at (-1,1.25) {$x_4'$};
					\node (x3) at (-2.4,3) {$x_3$};
					\node (x2) at (0.5,3) {$x_2$};
					\node (x1) at (-1.1,4.25) {$x_1$};
				\end{tikzpicture}\\
				&=\sum_{x_4' x_3'} \big(F_u^{x_3 x_2 c}\big)^*_{x_4' x_4}\ \big(F_{x_4'}^{a x_1 x_2}\big)_{x_3' x_3}\ 
				\begin{tikzpicture}[scale=.5,baseline=(current bounding box.center)]
					\draw (0,0) -- (0,1);
					\draw (0,1) -- (-1,2);
					\draw (0,1) -- (2,3);
					\draw (-1,2) -- (0,3);
					\draw (-1,2) -- (-2,3);
					\draw (0,3) -- (-1,4);
					\draw (0,3) -- (1,4);
					\draw (-1,4) -- (-1,5);
					\draw (1,4) -- (1,5);
					\draw (-1,5) -- (0,6);
					\draw (1,5) -- (0,6);
					\draw (-2,3) -- (-2,7);
					\draw (0,6) -- (0,7);
					\draw (2,3) -- (2,7);
					\node (u) at (0,-0.5) {$u$};
					\node (a) at (-2,7.5) {$a$};
					\node (b) at (0,7.5) {$b$};
					\node (c) at (2,7.5) {$c$};
					\node (x3p) at (0.1,2.25) {$x_3'$};
					\node (x1) at (-1.4,4) {$x_1$};
					\node (x2) at (1.5,4) {$x_2$};
					\node (x4p) at (-1,1.25) {$x_4'$};
				\end{tikzpicture}\\
				&=\sum_{x_4' x_3'} \big(F_u^{x_3 x_2 c}\big)^*_{x_4' x_4}\ \big(F_{x_4'}^{a x_1 x_2}\big)_{x_3' x_3}\ \sqrt{\frac{d_{x_1}d_{x_2}}{d_b}}\ \delta_{b, x_3'}\ 
				\begin{tikzpicture}[scale=.5,baseline=(current bounding box.center)]
					\node (u) at (6.5,-0.5) {$u$};
					\node (a) at (4.5,3.5) {$a$};
					\node (b) at (6.5,3.5) {$b$};
					\node (c) at (8.5,3.5) {$c$};
					\draw (u) -- (6.5,1);
					\draw (6.5,1) -- (8.5,3);
					\draw (5.5,2) -- (6.5,3);
					\draw (5.5,2) -- (4.5,3);
					\draw (6.5,1) -- (5.5,2) node[pos=0.4,left] {$x_4'\ $};
				\end{tikzpicture}\\
				&=\sum_{x_4'} \big(F_u^{x_3 x_2 c}\big)^*_{x_4' x_4}\ \big(F_{x_4'}^{a x_1 x_2}\big)_{b x_3}\ \sqrt{\frac{d_{x_1}d_{x_2}}{d_b}}\ 
				\begin{tikzpicture}[scale=.5,baseline=(current bounding box.center)]
					\node (u) at (6.5,-0.5) {$u$};
					\node (a) at (4.5,3.5) {$a$};
					\node (b) at (6.5,3.5) {$b$};
					\node (c) at (8.5,3.5) {$c$};
					\draw (u) -- (6.5,1);
					\draw (6.5,1) -- (8.5,3);
					\draw (5.5,2) -- (6.5,3);
					\draw (5.5,2) -- (4.5,3);
					\draw (6.5,1) -- (5.5,2) node[pos=0.4,left] {$x_4'\ $};
				\end{tikzpicture}.
			\end{align}
		}
		We can now compare these two expressions component wise, i.e.\ compare the coefficients of each basis vector that we indicate by $x_i=y$, which completes the proof:
			\begin{align}
				\sum_{x_3'}\big(F_u^{ax_1 x_4}\big)_{x_3' x_3}\ \big(F_{x_3'}^{x_1 x_2 c}\big)^*_{b x_4}\ \sqrt{\frac{d_{x_1}d_{x_2}}{d_b}}\ \big(F_u^{abc}\big)^*_{y x_3'}&=\big(F_u^{x_3 x_2 c}\big)^*_{y x_4}\ \big(F_{y}^{a x_1 x_2}\big)_{b x_3}\ \sqrt{\frac{d_{x_1}d_{x_2}}{d_b}}\\
				\sum_{x_3'}\big(F_u^{ax_1 x_4}\big)_{x_3' x_3}\ \big(F_{x_3'}^{x_1 x_2 c}\big)^*_{b x_4}\ \big(F_u^{abc}\big)^*_{y x_3'}&=\big(F_u^{x_3 x_2 c}\big)^*_{y x_4}\ \big(F_{y}^{a x_1 x_2}\big)_{b x_3}.\qedhere
			\end{align}
	\end{proof}
	
At first sight, it is not obvious in which way these additional equations are helpful for determining the $F$-matrices. It becomes clear when we consider the special case where one or more of the $F$-matrices on the left hand side are one-dimensional: in this case the sum on the left hand side vanishes. Additionally, in some cases the $F$-matrices are even equal to $1$ (see \cref{thm:triv_1dim,thm:eqto1_withlabels}) which further simplifies this equation.

\begin{thm}
	In the fusion category $\mathcal{H}_3$, the following equations hold:
	\begin{align}
	\begin{split}
	\label{eq_addtriv}
	\left(F_\rho^{\rho\rho\rho}\right)_{{}_\alpha\rho \rho}^*\left(F_{{}_\alpha\rho}^{\rho\rho\rho}\right)_{\rho\rho}\sqrt{d_\rho}&=c_1\left(F_\rho^{\rho\rho\rho}\right)_{{}_\alpha\rho 1}+c_2\left(F_\rho^{\rho\rho\rho}\right)_{{}_\alpha\rho \rho}\\
	\left(F_\rho^{\rho\rho\rho}\right)_{\alphastarrho \rho}^*\left(F_{\alphastarrho}^{\rho\rho\rho}\right)_{\rho\rho}\sqrt{d_\rho}&=c_1\left(F_\rho^{\rho\rho\rho}\right)_{\alphastarrho 1}+c_2\left(F_\rho^{\rho\rho\rho}\right)_{\alphastarrho \rho}.
	\end{split}
	\end{align}
\end{thm}

\begin{proof}
	We apply \cref{thm:addeq1} in the case $x_1=x_2=x_3=x_4=\rho$ and obtain
	\begin{equation}
	\label{eq:thmeq1}
	\begin{tikzpicture}[scale=.5,baseline=(current bounding box.center)]
	\draw (0,0) -- (0,1);
	\draw (0,1) -- (-1,2);
	\draw (0,1) -- (1,2);
	\draw (-1,2) -- (-2,3);
	\draw (-1,2) -- (0,3);
	\draw (1,2) -- (0,3);
	\draw (-2,3) -- (-2,4);
	\draw (0,3) -- (0,4);
	\draw (1,2) -- (2,3);
	\draw (2,3) -- (2,4);
	\node (1) at (0,-0.5) {$\rho$};
	\node (2) at (-2,4.5) {$\rho$};
	\node (3) at (0,4.5) {$\rho$};
	\node (4) at (-0.7,1.2) {$\rho$};
	\node (5) at (0.7,1.2) {$\rho$};
	\node (6) at (-0.7,2.7) {$\rho$};
	\node (7) at (0.7,2.7) {$\rho$};
	\node (8) at (2,4.5) {$\rho$};
	\end{tikzpicture}
	=\sum_{x} \big(F_\rho^{\rho\rho\rho}\big)^*_{x \rho}\ \big(F_{x}^{\rho\rho\rho}\big)_{\rho\rho}\ \sqrt{d_\rho}\ 
	\begin{tikzpicture}[scale=.5,baseline=(current bounding box.center)]
	\node (u) at (6.5,-0.5) {$\rho$};
	\node (a) at (4.5,3.5) {$\rho$};
	\node (b) at (6.5,3.5) {$\rho$};
	\node (c) at (8.5,3.5) {$\rho$};
	\draw (u) -- (6.5,1);
	\draw (6.5,1) -- (8.5,3);
	\draw (5.5,2) -- (6.5,3);
	\draw (5.5,2) -- (4.5,3);
	\draw (6.5,1) -- (5.5,2) node[pos=0.4,left] {$x\ $};
	\end{tikzpicture}.
	\end{equation}
	Additionally, we know from \cite{MPS17} that
	\begin{equation}
	\label{eq:SquarePop}
	\begin{tikzpicture}[scale=.8,baseline=(current bounding box.center)]
	\draw (-0.25,-0.25) -- (0,0);
	\draw (-0.25,1.25) -- (0,1);
	\draw (1.25,1.25) -- (1,1);
	\draw (1.25,-0.25) -- (1,0);	
	\draw (0,0) to [bend left] (0,1);
	\draw (0,1) to [bend left] (1,1);
	\draw (1,1) to [bend left] (1,0);
	\draw (1,0) to [bend left] (0,0);
	\end{tikzpicture}
	=c_1\ \left(\ 
	\begin{tikzpicture}[scale=.8,baseline=(current bounding box.center)]
	\draw (-0.25,-0.25) to [bend right=70] (-0.25,1.25);
	\draw (1.25,-0.25) to [bend left=70] (1.25,1.25);
	\end{tikzpicture}
	+
	\begin{tikzpicture}[scale=.8,baseline=(current bounding box.center)]
	\draw (-0.25,-0.25) to [bend left=70] (1.25,-0.25);
	\draw (-0.25,1.25) to [bend right=70] (1.25,1.25);
	\end{tikzpicture}
	\ \right)+c_2\ \left(\ 
	\begin{tikzpicture}[scale=.8,baseline=(current bounding box.center)]
	\draw (-0.25,-0.25) -- (0.25,0.5);
	\draw (-0.25,1.25) -- (0.25,0.5);
	\draw (0.25,0.5) -- (0.75,0.5);
	\draw (1.25,-0.25) -- (0.75,0.5);
	\draw (1.25,1.25) -- (0.75,0.5);
	\end{tikzpicture}
	+
	\begin{tikzpicture}[scale=.8,baseline=(current bounding box.center),rotate=90]
	\draw (-0.25,-0.25) -- (0.25,0.5);
	\draw (-0.25,1.25) -- (0.25,0.5);
	\draw (0.25,0.5) -- (0.75,0.5);
	\draw (1.25,-0.25) -- (0.75,0.5);
	\draw (1.25,1.25) -- (0.75,0.5);
	\end{tikzpicture}
	\ \right)
	\end{equation}
	with 
	\begin{align}
	c_1&=\frac{1}{18} \left(\sqrt{13}+7\right)\\
	c_2&=\frac{\sqrt{\sqrt{13}-2}}{3}
	\end{align}
	(see \cref{app:A_SquarePop} for the derivation of these constants). This equation can be translated into tree diagrams and then be manipulated by applying $F$-moves:
	\begin{align}
	\begin{tikzpicture}[scale=.4,baseline=(current bounding box.center)]
	\draw (0,0) -- (0,1);
	\draw (0,1) -- (-1,2);
	\draw (0,1) -- (1,2);
	\draw (-1,2) -- (-2,3);
	\draw (-1,2) -- (0,3);
	\draw (1,2) -- (0,3);
	\draw (-2,3) -- (-2,4);
	\draw (0,3) -- (0,4);
	\draw (1,2) -- (2,3);
	\draw (2,3) -- (2,4);
	\node (1) at (0,-0.5) {$\rho$};
	\node (2) at (-2,4.5) {$\rho$};
	\node (3) at (0,4.5) {$\rho$};
	\node (4) at (-0.7,1.2) {$\rho$};
	\node (5) at (0.7,1.2) {$\rho$};
	\node (6) at (-0.7,2.7) {$\rho$};
	\node (7) at (0.7,2.7) {$\rho$};
	\node (8) at (2,4.5) {$\rho$};
	\end{tikzpicture}
	&=c_1\ \left(
	\begin{tikzpicture}[scale=.4,baseline=(current bounding box.center)]
	\node (u) at (6.5,-0.5) {$\rho$};
	\node (a) at (4.5,3.5) {$\rho$};
	\node (b) at (6.5,3.5) {$\rho$};
	\node (c) at (8.5,3.5) {$\rho$};
	\draw (u) -- (6.5,1);
	\draw (6.5,1) -- (8.5,3);
	\draw (5.5,2) -- (6.5,3);
	\draw (5.5,2) -- (4.5,3);
	\draw (6.5,1) -- (5.5,2) node[pos=0.4,left] {$1\ $};
	\end{tikzpicture}
	+
	\begin{tikzpicture}[scale=.4,baseline=(current bounding box.center)]
	\node (u) at (6.5,-0.5) {$\rho$};
	\node (a) at (4.5,3.5) {$\rho$};
	\node (b) at (6.5,3.5) {$\rho$};
	\node (c) at (8.5,3.5) {$\rho$};
	\draw (u) -- (6.5,1);
	\draw (6.5,1) -- (7.5,2) node[pos=0.4,right] {$\ 1$};
	\draw (7.5,2) -- (6.5,3);
	\draw (7.5,2) -- (8.5,3);
	\draw (6.5,1) -- (4.5,3);
	\end{tikzpicture}
	\right) +c_2\ \left(
	\begin{tikzpicture}[scale=.4,baseline=(current bounding box.center)]
	\node (u) at (6.5,-0.5) {$\rho$};
	\node (a) at (4.5,3.5) {$\rho$};
	\node (b) at (6.5,3.5) {$\rho$};
	\node (c) at (8.5,3.5) {$\rho$};
	\draw (u) -- (6.5,1);
	\draw (6.5,1) -- (8.5,3);
	\draw (5.5,2) -- (6.5,3);
	\draw (5.5,2) -- (4.5,3);
	\draw (6.5,1) -- (5.5,2) node[pos=0.4,left] {$\rho\ $};
	\end{tikzpicture}
	+
	\begin{tikzpicture}[scale=.4,baseline=(current bounding box.center)]
	\node (u) at (6.5,-0.5) {$\rho$};
	\node (a) at (4.5,3.5) {$\rho$};
	\node (b) at (6.5,3.5) {$\rho$};
	\node (c) at (8.5,3.5) {$\rho$};
	\draw (u) -- (6.5,1);
	\draw (6.5,1) -- (7.5,2) node[pos=0.4,right] {$\ \rho$};
	\draw (7.5,2) -- (6.5,3);
	\draw (7.5,2) -- (8.5,3);
	\draw (6.5,1) -- (4.5,3);
	\end{tikzpicture}
	\right)\\
	\label{eq:thmeq2}
	&=c_1\ \left(
	\begin{tikzpicture}[scale=.4,baseline=(current bounding box.center)]
	\node (u) at (6.5,-0.5) {$\rho$};
	\node (a) at (4.5,3.5) {$\rho$};
	\node (b) at (6.5,3.5) {$\rho$};
	\node (c) at (8.5,3.5) {$\rho$};
	\draw (u) -- (6.5,1);
	\draw (6.5,1) -- (8.5,3);
	\draw (5.5,2) -- (6.5,3);
	\draw (5.5,2) -- (4.5,3);
	\draw (6.5,1) -- (5.5,2) node[pos=0.4,left] {$1\ $};
	\end{tikzpicture}
	+\sum_x \left(F_\rho^{\rho\rho\rho}\right)_{x1}\ 
	\begin{tikzpicture}[scale=.4,baseline=(current bounding box.center)]
	\node (u) at (6.5,-0.5) {$\rho$};
	\node (a) at (4.5,3.5) {$\rho$};
	\node (b) at (6.5,3.5) {$\rho$};
	\node (c) at (8.5,3.5) {$\rho$};
	\draw (u) -- (6.5,1);
	\draw (6.5,1) -- (8.5,3);
	\draw (5.5,2) -- (6.5,3);
	\draw (5.5,2) -- (4.5,3);
	\draw (6.5,1) -- (5.5,2) node[pos=0.4,left] {$x\ $};
	\end{tikzpicture}
	\right)\\
	&\hspace{40pt}+c_2\ \left(
	\begin{tikzpicture}[scale=.4,baseline=(current bounding box.center)]
	\node (u) at (6.5,-0.5) {$\rho$};
	\node (a) at (4.5,3.5) {$\rho$};
	\node (b) at (6.5,3.5) {$\rho$};
	\node (c) at (8.5,3.5) {$\rho$};
	\draw (u) -- (6.5,1);
	\draw (6.5,1) -- (8.5,3);
	\draw (5.5,2) -- (6.5,3);
	\draw (5.5,2) -- (4.5,3);
	\draw (6.5,1) -- (5.5,2) node[pos=0.4,left] {$\rho\ $};
	\end{tikzpicture}
	+\sum_x \left(F_\rho^{\rho\rho\rho}\right)_{x\rho}\ 
	\begin{tikzpicture}[scale=.4,baseline=(current bounding box.center)]
	\node (u) at (6.5,-0.5) {$\rho$};
	\node (a) at (4.5,3.5) {$\rho$};
	\node (b) at (6.5,3.5) {$\rho$};
	\node (c) at (8.5,3.5) {$\rho$};
	\draw (u) -- (6.5,1);
	\draw (6.5,1) -- (8.5,3);
	\draw (5.5,2) -- (6.5,3);
	\draw (5.5,2) -- (4.5,3);
	\draw (6.5,1) -- (5.5,2) node[pos=0.4,left] {$x\ $};
	\end{tikzpicture}
	\right)
	\end{align}
	Comparing the coefficients in \cref{eq:thmeq1} and \cref{eq:thmeq2} of the basis vectors for $x={}_\alpha\rho$ and $x=\alphastarrho$, respectively, completes the proof.
\end{proof}

\subsection{Further methods from trivalent categories}
\label{sec:MethodsTriv}

\begin{thm}
	\label{thm:triv_1dim}
	In the fusion category $\mathcal{H}_3$, if $a,b,c,u\in\{1,\rho\}$ and at least one of the labels $a$, $b$, $c$, or $u$ is the unit label $1$ and the fusion of $a$, $b$, and $c$ to the label $u$ is allowed by the fusion rules, the vector space $V_u^{abc}$ is one-dimensional and the corresponding $F$-matrix is 
	\begin{equation}
	F_u^{abc}=1.
	\end{equation}
\end{thm}
\begin{proof}
	Let, without loss of generality, $a=1$. Then we can decompose the vector space $V_u^{1bc}$ into
	\begin{align}
	V_u^{1bc}&=\bigoplus_x V_x^{1b}\otimes V_u^{xc}\\
	&=V_b^{1b}\otimes V_u^{bc},
	\end{align}
	therefore $\dim \left(V_u^{1bc}\right)=\dim \left(V_b^{1b}\otimes V_u^{bc}\right)= N_b^{1b} N_u^{bc}=1$. The $F$-matrix acts on this decomposition in the following way:
	\begin{align}
	\begin{tikzpicture}[scale=.5,baseline=(current bounding box.center)]
	\node (u) at (0,-0.5) {$u$};
	\node (1) at (-2,3.5) {$1$};
	\node (b) at (0,3.5) {$b$};
	\node (c) at (2,3.5) {$c$};
	\draw (u) -- (0,1);
	\draw (0,1) -- (-1,2) node[pos=0.4,left] {$b$};
	\draw[dashed] (-1,2) -- (-2,3);
	\draw (-1,2) -- (0,3);
	\draw (0,1) -- (2,3);
	\end{tikzpicture}
	&=\sum_y \left(F_u^{1bc}\right)_{by}
	\begin{tikzpicture}[scale=.5,baseline=(current bounding box.center)]
	\node (u) at (6.5,-0.5) {$u$};
	\node (1) at (4.5,3.5) {$1$};
	\node (b) at (6.5,3.5) {$b$};
	\node (c) at (8.5,3.5) {$c$};
	\draw (u) -- (6.5,1);
	\draw (6.5,1) -- (7.5,2) node[pos=0.4,right] {$y$};
	\draw (7.5,2) -- (6.5,3);
	\draw (7.5,2) -- (8.5,3);
	\draw[dashed] (6.5,1) -- (4.5,3);
	\end{tikzpicture}\\
	&=\left(F_u^{1bc}\right)_{bu}
	\begin{tikzpicture}[scale=.5,baseline=(current bounding box.center)]
	\node (u) at (6.5,-0.5) {$u$};
	\node (1) at (4.5,3.5) {$1$};
	\node (b) at (6.5,3.5) {$b$};
	\node (c) at (8.5,3.5) {$c$};
	\draw (u) -- (6.5,1);
	\draw (6.5,1) -- (7.5,2) node[pos=0.4,right] {$u$};
	\draw (7.5,2) -- (6.5,3);
	\draw (7.5,2) -- (8.5,3);
	\draw[dashed] (6.5,1) -- (4.5,3);
	\end{tikzpicture}
	\end{align}
	According to the triangle equation,  
	\begin{align}
	\l_b
	\begin{tikzpicture}[scale=.5,baseline=(current bounding box.center)]
	\draw (0,0) -- (0,1);
	\draw (0,1) -- (-1,2);
	\draw (0,1) -- (1,2);
	\node (u) at (0,-0.5) {$u$};
	\node (b) at (-1,2.5) {$b$};
	\node (c) at (1,2.5) {$c$};
	\end{tikzpicture}
	=\left(F_u^{1bc}\right)_{bu}\ \l_u
	\begin{tikzpicture}[scale=.5,baseline=(current bounding box.center)]
	\draw (0,0) -- (0,1);
	\draw (0,1) -- (-1,2);
	\draw (0,1) -- (1,2);
	\node (u) at (0,-0.5) {$u$};
	\node (b) at (-1,2.5) {$b$};
	\node (c) at (1,2.5) {$c$};
	\end{tikzpicture}.
	\end{align}
	Since $b,c,u\in\{1,\rho\}$, only the maps $l_1$ and $l_\rho$ can occur. From \cite{MPS17} we know that $l_1=l_\rho=1$. Hence,
	\begin{align}
	\begin{tikzpicture}[scale=.5,baseline=(current bounding box.center)]
	\draw (0,0) -- (0,1);
	\draw (0,1) -- (-1,2);
	\draw (0,1) -- (1,2);
	\node (u) at (0,-0.5) {$u$};
	\node (b) at (-1,2.5) {$b$};
	\node (c) at (1,2.5) {$c$};
	\end{tikzpicture}
	=\left(F_u^{1bc}\right)_{bu}
	\begin{tikzpicture}[scale=.5,baseline=(current bounding box.center)]
	\draw (0,0) -- (0,1);
	\draw (0,1) -- (-1,2);
	\draw (0,1) -- (1,2);
	\node (u) at (0,-0.5) {$u$};
	\node (b) at (-1,2.5) {$b$};
	\node (c) at (1,2.5) {$c$};
	\end{tikzpicture}
	\end{align}
	and therefore $\left(F_u^{1bc}\right)_{bu}=1$. The proof for the cases $b=1$ and $c=1$ works analogously.
\end{proof}

\begin{thm}
	\label{thm:eqto1_withlabels}
	In the fusion category $\mathcal{H}_3$,
	\begin{align}
	F_x^{\rho1\rho}&=1,\\
	F_\rho^{1\rho x}&=1,\\
	F_\rho^{x \rho1}&=1,
	\end{align}
	for $x\in\{\rho,{}_\alpha\rho,\alphastarrho\}$.
\end{thm}
\begin{proof}
	This result is a combination of special cases of the triangle equation and properties that follow from the fact that $\mathcal{H}_3$ is a trivalent category. The first equations results from the following specific triangle equation 
	\begin{figure}[H]
		\centering
		\begin{tikzpicture}[scale=0.4]
		\node at (0,-1.4) {$x$};
		\node at (-1,1.4) {$\rho$};
		\node at (1,1.4) {$\rho$};
		\draw (0,-1) -- (0,0);
		\draw (0,0) -- (-1,1);
		\draw (0,0) -- (1,1);
		\node at (-5,4.6) {$x$};
		\node at (-7,8.4) {$\rho$};
		\node at (-5,8.4) {$1$};
		\node at (-3,8.4) {$\rho$};
		\draw (-5,5) -- (-5,6);
		\draw (-5,6) -- (-6,7);
		\draw (-5,6) -- (-4,7);
		\draw (-6,7) -- (-7,8);
		\draw[dashed] (-6,7) -- (-5,8);
		\draw (-4,7) -- (-3,8);
		\node at (5,4.6) {$x$};
		\node at (7,8.4) {$\rho$};
		\node at (5,8.4) {$1$};
		\node at (3,8.4) {$\rho$};
		\draw (5,5) -- (5,6);
		\draw (5,6) -- (6,7);
		\draw (5,6) -- (4,7);
		\draw (6,7) -- (7,8);
		\draw[dashed] (6,7) -- (5,8);
		\draw (4,7) -- (3,8);
		\draw[->] (-2,6) -- node[above] {$F_x^{\rho1\rho}$} (2,6);
		\draw[->] (-4,4) -- node[left] {$r_\rho$} (-2,2);
		\draw[->] (4,4) -- node[right] {$\ l_\rho$} (2,2);
		\end{tikzpicture}
	\end{figure}
	\noindent
	Since every map in this diagram is one-dimensional and we know that $r_\rho=l_\rho=1$ (since it is a trivalent category),
	\begin{align}
	r_\rho&=F_x^{\rho1\rho}l_\rho\\
	1&=F_x^{\rho1\rho}.
	\end{align}
	The remaining equations follow analogously as special cases of the variations of the triangle equation shown in \cref{fig:TriangleCases}.
\end{proof}

\begin{thm}
	\label{thm:Frrrr}
	In the fusion category $\mathcal{H}_3$, the following equation holds:
	\begin{align}
	\big(F_\rho^{\rho\rho\rho}\big)_{\rho\rho}=\frac{t}{\sqrt{d_\rho}}
	\end{align}
	with $t=-\left(\frac{2}{3}d_\rho+\frac{5}{3}\right)\sqrt{d_\rho}$ and $d_\rho=\frac{3+\sqrt{13}}{2}$.
\end{thm}
\begin{proof}
	From~\cite{MPS17}, we have the following relation:
	\begin{equation}
	\begin{tikzpicture}[scale=.3,baseline=(current bounding box.center)]
	\draw (0,0) -- (0,1);
	\draw (0,1) to [bend left] (-1.5,3.5);
	\draw (0,1) to [bend right] (1.5,3.5);
	\draw (-1.5,3.5) to [bend left] (1.5,3.5);
	\draw (-1.5,3.5) to (-2.5,4.5);
	\draw (1.5,3.5) to (2.5,4.5);
	\end{tikzpicture}
	= t\ 
	\begin{tikzpicture}[scale=.3,baseline=(current bounding box.center)]
	\draw (0,0) -- (0,2);
	\draw (0,2) -- (-2,4);
	\draw (0,2) -- (2,4);
	\end{tikzpicture}.
	\end{equation}
	The absence of labels in the strings indicates that every string stands for the generating object of the category, $\rho$ in our case. Since in this category, these diagrams are invariant under rotation and deforming, and we can add and remove strings with the unit label on it, we can translate these vectors into vectors that live in the space
	\begin{equation}
	V_\rho^{\rho\rho1}=\bigoplus_{x_1,x_2,x_3,x_4}V_\rho^{x_3 x_4}\otimes V_{x_3}^{\rho x_1}\otimes V_{x_4}^{x_2 1}\otimes V_{x_1 x_2}^{\rho}.
	\end{equation}
	The whole relation can now be expressed as an equation in this vector space:
	\begin{equation}
	\begin{tikzpicture}[scale=.5,baseline=(current bounding box.center)]
	\draw (0,0) -- (0,1);
	\draw (0,1) -- (-1,2);
	\draw (0,1) -- (1,2);
	\draw (-1,2) -- (-2,3);
	\draw (-1,2) -- (0,3);
	\draw (1,2) -- (0,3);
	\draw (-2,3) -- (-2,4);
	\draw (0,3) -- (0,4);
	\draw (1,2) -- (2,3);
	\draw (2,3) -- (2,4);
	\node (1) at (0,-0.5) {$\rho$};
	\node (2) at (-2,4.5) {$\rho$};
	\node (3) at (0,4.5) {$\rho$};
	\node (4) at (-0.7,1.2) {$\rho$};
	\node (5) at (0.7,1.2) {$\rho$};
	\node (6) at (-0.7,2.7) {$\rho$};
	\node (7) at (0.7,2.7) {$\rho$};
	\node (8) at (2,4.5) {$1$};
	\end{tikzpicture}
	=t\ 
	\begin{tikzpicture}[scale=.5,baseline=(current bounding box.center)]
	\node (u) at (6.5,-0.5) {$\rho$};
	\node (a) at (4.5,3.5) {$\rho$};
	\node (b) at (6.5,3.5) {$\rho$};
	\node (c) at (8.5,3.5) {$1$};
	\draw (u) -- (6.5,1);
	\draw (6.5,1) -- (8.5,3);
	\draw (5.5,2) -- (6.5,3);
	\draw (5.5,2) -- (4.5,3);
	\draw (6.5,1) -- (5.5,2) node[pos=0.4,left] {$\rho\ $};
	\end{tikzpicture}.
	\label{Eq:t}
	\end{equation}
	Using the relations found in the proof of Theorem~\ref{thm:addeq1}, we find that
	\begin{equation}
	\begin{tikzpicture}[scale=.5,baseline=(current bounding box.center)]
	\draw (0,0) -- (0,1);
	\draw (0,1) -- (-1,2);
	\draw (0,1) -- (1,2);
	\draw (-1,2) -- (-2,3);
	\draw (-1,2) -- (0,3);
	\draw (1,2) -- (0,3);
	\draw (-2,3) -- (-2,4);
	\draw (0,3) -- (0,4);
	\draw (1,2) -- (2,3);
	\draw (2,3) -- (2,4);
	\node (1) at (0,-0.5) {$\rho$};
	\node (2) at (-2,4.5) {$\rho$};
	\node (3) at (0,4.5) {$\rho$};
	\node (4) at (-0.7,1.2) {$\rho$};
	\node (5) at (0.7,1.2) {$\rho$};
	\node (6) at (-0.7,2.7) {$\rho$};
	\node (7) at (0.7,2.7) {$\rho$};
	\node (8) at (2,4.5) {$1$};
	\end{tikzpicture}
	=\sum_{x} \big(F_\rho^{\rho\rho1}\big)^*_{x \rho}\ \big(F_{x}^{\rho\rho\rho}\big)_{\rho\rho}\ \sqrt{d_\rho}\ 
	\begin{tikzpicture}[scale=.5,baseline=(current bounding box.center)]
	\node (u) at (6.5,-0.5) {$\rho$};
	\node (a) at (4.5,3.5) {$\rho$};
	\node (b) at (6.5,3.5) {$\rho$};
	\node (c) at (8.5,3.5) {$1$};
	\draw (u) -- (6.5,1);
	\draw (6.5,1) -- (8.5,3);
	\draw (5.5,2) -- (6.5,3);
	\draw (5.5,2) -- (4.5,3);
	\draw (6.5,1) -- (5.5,2) node[pos=0.4,left] {$x\ $};
	\end{tikzpicture}.
	\end{equation}
	The vector on the right hand side of the equation is an element of the vector space $\bigoplus_x V_x^{\rho\rho}\otimes V_\rho^{x1}$ which is only non-zero for $x=\rho$. Additionally, we know from Theorem \ref{thm:triv_1dim} that $F_\rho^{\rho\rho1}=1$. Hence, the equation is reduced to
	\begin{equation}
	\begin{tikzpicture}[scale=.5,baseline=(current bounding box.center)]
	\draw (0,0) -- (0,1);
	\draw (0,1) -- (-1,2);
	\draw (0,1) -- (1,2);
	\draw (-1,2) -- (-2,3);
	\draw (-1,2) -- (0,3);
	\draw (1,2) -- (0,3);
	\draw (-2,3) -- (-2,4);
	\draw (0,3) -- (0,4);
	\draw (1,2) -- (2,3);
	\draw (2,3) -- (2,4);
	\node (1) at (0,-0.5) {$\rho$};
	\node (2) at (-2,4.5) {$\rho$};
	\node (3) at (0,4.5) {$\rho$};
	\node (4) at (-0.7,1.2) {$\rho$};
	\node (5) at (0.7,1.2) {$\rho$};
	\node (6) at (-0.7,2.7) {$\rho$};
	\node (7) at (0.7,2.7) {$\rho$};
	\node (8) at (2,4.5) {$1$};
	\end{tikzpicture}
	=\big(F_\rho^{\rho\rho\rho}\big)_{\rho \rho}\ \sqrt{d_\rho}
	\begin{tikzpicture}[scale=.5,baseline=(current bounding box.center)]
	\node (u) at (6.5,-0.5) {$\rho$};
	\node (a) at (4.5,3.5) {$\rho$};
	\node (b) at (6.5,3.5) {$\rho$};
	\node (c) at (8.5,3.5) {$1$};
	\draw (u) -- (6.5,1);
	\draw (6.5,1) -- (8.5,3);
	\draw (5.5,2) -- (6.5,3);
	\draw (5.5,2) -- (4.5,3);
	\draw (6.5,1) -- (5.5,2) node[pos=0.4,left] {$\rho\ $};
	\end{tikzpicture}.
	\end{equation}
	Comparing this result with \cref{Eq:t} yields the statement.  
\end{proof}

\subsection{Unitarity}

Since the $F$-matrices are basis transformations, we require them to be unitary. Hence, we can use the unitarity condition
\begin{equation}
U^*U=UU^*=\mathbb{I}
\end{equation}
to get additional equations for the variables. This is especially helpful in case we have already determined the majority of variables for one matrix: by using unitarity, it is likely that we can also determine the rest of them. For a one-dimensional matrix $F_u^{abc}$, unitarity implies
\begin{equation}
\left\lVert F_u^{abc}\right\rVert^2=1.
\end{equation}

\section{Gauge freedom}
\label{sec:6_gauge}

As explained in \cite{Bonderson}, there is a gauge freedom assigned to each distinct vertex that amounts to the choice of basis vertex. For $\psi\in V_c^{ab}$, let $u_c^{ab}$ be the invertible change of basis transformation for this space, i.e.\
	\begin{equation}
		\psi=u_c^{ab}\psi'.
	\end{equation}
For an element of an $F$-matrix, this takes the form
	\begin{equation}
	\label{eq:gauge}
		\left(F_d^{abc}\right)_{fe}'=\frac{u_d^{af} u_f^{bc}}{u_e^{ab} u_d^{ec}}\left(F_d^{abc}\right)_{fe}.
	\end{equation}
The corresponding gauge transformation for the adjoint $F$-matrix is
	\begin{equation}
	\label{eq:adjgauge}
		\left(\left[F_d^{abc}\right]^*\right)_{fe}'=\frac{u_e^{ab} u_d^{ec}}{u_d^{af} u_f^{bc}}\left(F_d^{abc}\right)_{fe}.
	\end{equation}

The strategy we pursue for fixing the gauge is as follows: We start by fixing the ratios to $1$ for those variables that have already been determined by one of the theorems we have shown above. We leave the remaining freedoms unfixed until we get to a point within the set of equations where fixing one freedom helps determining a variety of $F$-symbols.

\begin{thm}
	The pentagon equations \cref{eq:pentagon}, the equations following from the triangle equation \cref{eq:triangle} and the additional equations \cref{eq:addeq} are invariant under the gauge transformations \cref{eq:gauge,eq:adjgauge}.
\end{thm}		
\begin{proof} We prove gauge invariance for each of the equations separately:
	\begin{enumerate}
		\item The pentagon equations: under a gauge transformation, the pentagon equations \cref{eq:pentagon} are
				\begin{align}
					\frac{u_u^{xd}u_d^{yc}}{u_a^{xy}u_u^{ac}}\left(F_u^{xyc}\right)_{da}\frac{u_u^{ac}u_c^{zw}}{u_b^{az}u_u^{bw}}\left(F_u^{azw}\right)_{cb}&=\sum_e\frac{u_d^{yc}u_c^{zw}}{u_e^{yz}u_d^{ew}}\left(F_d^{yzw}\right)_{ce}\frac{u_u^{xd}u_d^{bw}}{u_b^{xe}u_u^{bw}}\left(F_u^{xew}\right)_{db}\frac{u_b^{xe}u_e^{yz}}{u_a^{xy}u_b^{az}}\left(F_b^{xyz}\right)_{ea}\\
					\frac{u_u^{xd}u_d^{yc}u_u^{ac}u_c^{zw}}{u_a^{xy}u_u^{ac}u_b^{az}u_u^{bw}}\left(F_u^{xyc}\right)_{da}\left(F_u^{azw}\right)_{cb}&=\frac{u_d^{yc}u_c^{zw}u_u^{xd}u_d^{bw}}{u_b^{xe}u_u^{bw}u_a^{xy}u_b^{az}}\ \sum_e\left(F_d^{yzw}\right)_{ce}\left(F_u^{xew}\right)_{db}\left(F_b^{xyz}\right)_{ea}\\
					\left(F_u^{xyc}\right)_{da}\left(F_u^{azw}\right)_{cb}&=\sum_e\left(F_d^{yzw}\right)_{ce}\left(F_u^{xew}\right)_{db}\left(F_b^{xyz}\right)_{ea},
				\end{align}
			therefore it is gauge invariant.
		\item The triangle equation: under a gauge transformation, the first equation in \cref{eq:triangle} has the form 
				\begin{align}
					\frac{u_z^{1 z} u_z^{xy}}{u_x^{1 x} u_z^{xy}} \left(F_z^{1 x y}\right)_{zx}\ \frac{u_x^{1 x} u_x^{zy}}{u_z^{1 z} u_x^{zy}} \left(F_x^{1 z y}\right)_{xz}&=1\\
					\left(F_z^{1 x y}\right)_{zx}\ \left(F_x^{1 z y}\right)_{xz}&=1,
				\end{align}
			hence it is gauge invariant. This works analogously for the second equation:
				\begin{align}
					\frac{u_x^{yz} u_y^{y1}}{u_x^{yx} u_z^{z1}} \left(F_z^{xy 1}\right)_{yz}\ \frac{u_y^{xz} u_z^{z1}}{u_y^{xz} u_y^{y 1}} \left(F_y^{xz1}\right)_{zy}&=1\\
					\left(F_z^{xy 1}\right)_{yz}\ \left(F_y^{xz1}\right)_{zy}&=1.
				\end{align}
		\item Additional equations: under a gauge transformation, \cref{eq:addeq} becomes
				\begin{align}
					\sum_{x_3'}\frac{u_u^{a x_3'} u_{x_3'}^{x_1 x_4}}{u_{x_3}^{a x_1} u_u^{x_3 x_4}}&\big(F_u^{ax_1 x_4}\big)_{x_3' x_3}\ \frac{u_b^{x_1 x_2} u_{x_3'}^{bc}}{u_{x_3'}^{x_1 x_4} u_{x_4}^{x_2 c}}\big(F_{x_3'}^{x_1 x_2 c}\big)^*_{b x_4}\ \frac{u_y^{ab} u_u^{yc}}{u_u^{a x_3'} u_{x_3'}^{bc}}\big(F_u^{abc}\big)^*_{y x_3'}\\
					&=\frac{u_y^{x_3 x_2} u_u^{yc}}{u_u^{x_3 x_4} u_{x_4}^{x_2 c}}\big(F_u^{x_3 x_2 c}\big)^*_{y x_4}\ \frac{u_y^{ab} u_b^{x_1 x_2}}{u_{x_3}^{a x_1} u_y^{x_3 x_2}}\big(F_{y}^{a x_1 x_2}\big)_{b x_3},
				\end{align}
				which yields 
				\begin{align}
					\frac{u_b^{x_1 x_2} u_y^{ab} u_u^{yc}}{u_{x_3}^{a x_1} u_u^{x_3 x_4}u_{x_4}^{x_2 c}}\sum_{x_3'}\big(F_u^{ax_1 x_4}\big)_{x_3' x_3}\big(F_{x_3'}^{x_1 x_2 c}\big)^*_{b x_4}\big(F_u^{abc}\big)^*_{y x_3'}&=\frac{u_u^{yc}u_y^{ab} u_b^{x_1 x_2}}{u_u^{x_3 x_4} u_{x_4}^{x_2 c}u_{x_3}^{a x_1}}\big(F_u^{x_3 x_2 c}\big)^*_{y x_4}\big(F_{y}^{a x_1 x_2}\big)_{b x_3}\\
					\sum_{x_3'}\big(F_u^{ax_1 x_4}\big)_{x_3' x_3}\big(F_{x_3'}^{x_1 x_2 c}\big)^*_{b x_4}\big(F_u^{abc}\big)^*_{y x_3'}&=\big(F_u^{x_3 x_2 c}\big)^*_{y x_4}\big(F_{y}^{a x_1 x_2}\big)_{b x_3}.\qedhere
				\end{align}
	\end{enumerate}
\end{proof}
\begin{rem}
	Note that \cref{eq_addtriv} is not invariant under the gauge transformations \cref{eq:gauge,eq:adjgauge}. Therefore, we have to keep track of the occurring gauge factors and adjust this equation whenever one of the factors is fixed.
\end{rem}


\section{Conclusion}
\label{sec:7_conclusion}

In this paper, we have presented a solution for the $F$-symbols of the $\mathcal{H}_3$ fusion category. The methods we have used can easily be generalised to other fusion categories that provide some skein-theoretic information, which is especially helpful in cases where the number of equations and unknowns is beyond what can be solved easily with current state-of-the-art technology and algorithms. An interesting candidate for this is the extended Haagerup subfactor, where the $F$-symbols could not be determined so far.

Furthermore, there are several interesting tasks that can be done now that the $F$-symbols for this category are known. One example is the Levin-Wen string-net model \cite{LevinWen}, whose construction requires a solution for the $F$-symbols and which yields interesting insights into the nature of the corresponding CFT. It is now also possible to construct the quantum double of $\mathcal{H}_3$ (see, e.g., \cite{Mue2}), which is, in contrast to $\mathcal{H}_3$, a braided tensor category.

\section*{Acknowledgements}
Helpful correspondence and discussions with Andreas Bluhm, Markus Duwe, Terry Farrelly, Alexander Hahn, Nathan McMahon, Scott Morrison, and Noah Snyder are gratefully acknowledged.
This work was supported by the DFG through SFB 1227 (DQ-mat), the RTG 1991, and the cluster of excellence EXC 2123 QuantumFrontiers.


\bibliographystyle{alpha}
\bibliography{lit}


\appendix

\section{Diagrammatical calculus of trivalent categories}
\label{app:TriCats}

\subsection{Evaluating diagrams}
\label{app:A_DiagramCalc}
We will now explain how to simplify and evaluate trivalent diagrams. For diagrams with up to three boundary points, the definition of a trivalent category directly provides the respective simplification rules.

Because of $\dim\mathcal{C}_0=1$, every diagram with no boundary points is a scalar multiple of the empty diagram. Therefore, the loop is a non-zero complex number which we denote $d$, i.e.
\begin{equation}
\begin{tikzpicture}[scale=.8,baseline=(current bounding box.center)]
\draw (0,0) circle (0.6cm);
\end{tikzpicture}
=d.
\end{equation}
In $\mathcal{H}_3$, $d=d_\rho=\frac{3+\sqrt{13}}{2}$. From $\dim\mathcal{C}_1=0$, it follows that
\begin{equation}
\begin{tikzpicture}[scale=.8,baseline=(current bounding box.center)]
\draw (0,0) circle (0.4cm);
\draw (0,-0.4) -- (0,-1);
\end{tikzpicture}
=0.
\end{equation}
Furthermore, $\dim\mathcal{C}_2=0$ implies 
\begin{equation}
\label{eq:bigon}
\begin{tikzpicture}[scale=.8,baseline=(current bounding box.center)]
\draw (0,0) -- (0,0.5);
\draw (0,0.5) to [bend left=70] (0,1.5);
\draw (0,0.5) to [bend right=70] (0,1.5);
\draw (0,1.5) -- (0,2);
\end{tikzpicture}
=b\ \ \begin{tikzpicture}[scale=.8,baseline=(current bounding box.center)]
\draw (0,0) -- (0,2);
\end{tikzpicture}
\end{equation}
where $b$ is another non-zero complex constant. Because one can always rescale the trivalent vertex by a constant, we can, without loss of generality, fix the value of $b$. For our computations, we use $b=\sqrt{d_\rho}$, which has the reason that in the full fusion category, we have the relation
\begin{equation}
\label{eq:bigon_norm}
\begin{tikzpicture}[scale=.9,baseline=(current bounding box.center)]
\draw (0,0) -- (0,0.5);
\draw (0,0.5) to [bend left=70] (0,1.5);
\draw (0,0.5) to [bend right=70] (0,1.5);
\draw (0,1.5) -- (0,2);
\node at (-0.2,0.2) {$c$};
\node at (-0.2,1.8) {$c'$};
\node at (-0.5,1) {$a$};
\node at (0.5,1) {$b$};
\end{tikzpicture}
=\delta_{cc'}\sqrt{\frac{d_a d_b}{d_c}}\ \ \begin{tikzpicture}[scale=.8,baseline=(current bounding box.center)]
\draw (0,0) -- (0,2);
\node at (0.2,0.3) {$c$};
\end{tikzpicture}.
\end{equation}
When we choose $c=c'=a=b=\rho$, we see that $b=\sqrt{d_\rho}$ for \cref{eq:bigon} and \cref{eq:bigon_norm} to be equal. Finally, $\dim\mathcal{C}_3=1$ means that
\begin{equation}
\begin{tikzpicture}[scale=.3,baseline=(current bounding box.center)]
\draw (0,0) -- (0,1);
\draw (0,1) to [bend left] (-1.5,3.5);
\draw (0,1) to [bend right] (1.5,3.5);
\draw (-1.5,3.5) to [bend left] (1.5,3.5);
\draw (-1.5,3.5) to (-2.5,4.5);
\draw (1.5,3.5) to (2.5,4.5);
\end{tikzpicture}
= t\ 
\begin{tikzpicture}[scale=.3,baseline=(current bounding box.center)]
\draw (0,0) -- (0,2);
\draw (0,2) -- (-2,4);
\draw (0,2) -- (2,4);
\end{tikzpicture},
\end{equation}
where $t$, in general, is a complex number that can be zero. In $\mathcal{H}_3$, $t=\left(-\frac{2}{3} d_\rho+\frac{5}{3}\right)\sqrt{d_\rho}$.

\subsection{Square popping relation}
\label{app:A_SquarePop}
Here, we derive the square popping relation \cref{eq:SquarePop} for the case $b\neq1$. In \cite{MPS17}, the calculation was only done for $b=1$, so we will do it for the general case in the following. 

To calculate relations between diagrams in $\mathcal{C}_4$, we need an orthonormal basis for this space, which we will obtain using the Gram-Schmidt process. We begin with the following (non-orthonormal) basis for $\mathcal{C}_4$:
	\begin{align}
		w_1&=
		\begin{tikzpicture}[scale=.7,baseline=(current bounding box.center)]
			\draw (0,0) arc (-180:0:1);
			\draw (0.5,0) arc (-180:0:0.5);
		\end{tikzpicture}
		&&w_2=\begin{tikzpicture}[scale=.7,baseline=(current bounding box.center)]
			\draw (0,0) arc (-180:0:0.4);
			\draw (1.2,0) arc (-180:0:0.4);
		\end{tikzpicture}\\
		w_3&=\begin{tikzpicture}[scale=.7,baseline=(current bounding box.center)]
			\draw (0,0) arc (-180:0:0.4);
			\draw (1.2,0) arc (-180:0:0.4);
			\draw (0.4,-0.4) arc (-180:0:0.6);
		\end{tikzpicture}
		&&w_4=
		\begin{tikzpicture}[scale=.7,baseline=(current bounding box.center)]
			\draw (0,0) arc (-180:0:1);
			\draw (0.5,0) arc (-180:0:0.5);
			\draw (1,-1) -- (1,-0.5);
		\end{tikzpicture}
	\end{align}
	Using the Gram-Schmidt process, we can find a matrix $\Theta$ such that the orthonormal basis vectors $v_1,\dots,v_4$ are of the form
		\begin{equation}
			v_i=\sum_k \Theta_{ik} w_k.
		\end{equation}
	This matrix is of the following form:
		\begin{equation}\Theta=
			\begin{pmatrix}
				\frac{1}{d} & 0 & 0 & 0 \\
				-\frac{1}{d\sqrt{d^2-1}} & \frac{1}{\sqrt{d^2-1}} & 0 & 0 \\
				-\frac{b \sqrt{d}}{\sqrt{(d^2-a)(d^2-d-)}} & \frac{b}{\sqrt{(d^2-a)(d^2-d-)}} & \sqrt{\frac{d^2-1}{b^2d(d^2-d-1)}} & 0 \\
				\frac{b+dt}{d^2-d-1} C_\Theta & \frac{b-bd-t}{d^2-d-1} C_\Theta & \frac{-d^2t+t-b}{b(d^2-d-1)} C_\Theta & C_\Theta
			\end{pmatrix}
		\end{equation}
	with
		\begin{equation}
			C_\Theta=\sqrt{\frac{d^2-d-1}{d(b^2d(d-2)-2bt-(d^2-1)t^2}}.
		\end{equation}
	Using this orthonormal basis, we can calculate the coefficients of the square in this basis, hence we can also express the square in terms of the original basis $w_1,\dots,w_4$. The resulting formula is
		\begin{equation}
			\begin{tikzpicture}[scale=.7,baseline=(current bounding box.center)]
				\draw (0,0) arc (-180:0:1);
				\draw (0.5,0) arc (-180:0:0.5);
				\draw (0.6,-0.3) -- (0.35,-0.75);
				\draw (1.4,-0.3) -- (1.65,-0.75);
			\end{tikzpicture}=\frac{b(b^2+bt-t^2)}{bd+t+dt}\left(
			\begin{tikzpicture}[scale=.7,baseline=(current bounding box.center)]
				\draw (0,0) arc (-180:0:1);
				\draw (0.5,0) arc (-180:0:0.5);
			\end{tikzpicture}+
			\begin{tikzpicture}[scale=.7,baseline=(current bounding box.center)]
				\draw (0,0) arc (-180:0:0.4);
				\draw (1.2,0) arc (-180:0:0.4);
			\end{tikzpicture}\right)
			+\frac{t^2(d+1)-b^2}{bd+t+dt}\left(
			\begin{tikzpicture}[scale=.7,baseline=(current bounding 	box.center)]
				\draw (0,0) arc (-180:0:0.4);
				\draw (1.2,0) arc (-180:0:0.4);
				\draw (0.4,-0.4) arc (-180:0:0.6);
			\end{tikzpicture}+
			\begin{tikzpicture}[scale=.7,baseline=(current bounding box.center)]
				\draw (0,0) arc (-180:0:1);
				\draw (0.5,0) arc (-180:0:0.5);
				\draw (1,-1) -- (1,-0.5);
			\end{tikzpicture}\right).
		\end{equation}
The source code for the calculations can be found in \cite{Sti18b,Sti18a}.

\section{$F$-symbols for the $\mathcal{H}_3$ fusion category}
\label{app:Fsymbols}
This solution and the code which was used to calculate it can be found under \url{https://github.com/R8monaW/H3Fsymbols}. It is a real solution with two parameters, $p_1,p_2\in\{-1,+1\}$. Additionally, we use the following notation:
	\begin{align}
		A&=\frac{1}{2} \left(\sqrt{13}-3\right),\\
		B&=\frac{1}{3} \left(\sqrt{13}-2\right),\\
		C&=\frac{1}{6} \left(\sqrt{13}+1\right),\\
		D_{\pm}&=\frac{1}{12} \left(5-\sqrt{13}\pm\sqrt{6 \left(\sqrt{13}+1\right)}\right).
	\end{align}
In the following, all $F$-symbols are listed. They are first ordered by their dimension and within that, they are ordered by their labels. The $3$- and $4$-dimensional matrices are depicted as tables where the first row and the first column indicate the admissible labels for the matrix.
\subsection{$1$-dimensional $F$-symbols}
{\allowdisplaybreaks\begin{align}
F_{1}^{111}&=1&F_{1}^{1\alpha\alpha^*}&=1&F_{1}^{1\alpha^*\alpha}&=1&F_{1}^{1\rho\rho}&=1\\F_{1}^{1{}_\alpha\rho{}_\alpha\rho}&=1&F_{1}^{1\alphastarrho\alphastarrho}&=1&F_{1}^{\alpha1\alpha^*}&=1&F_{1}^{\alpha\alpha\alpha}&=1\\F_{1}^{\alpha\alpha^*1}&=1&F_{1}^{\alpha\rho{}_\alpha\rho}&=1&F_{1}^{\alpha{}_\alpha\rho\alphastarrho}&=-p_1&F_{1}^{\alpha\alphastarrho\rho}&=1\\F_{1}^{\alpha^*1\alpha}&=1&F_{1}^{\alpha^*\alpha1}&=1&F_{1}^{\alpha^*\alpha^*\alpha^*}&=1&F_{1}^{\alpha^*\rho\alphastarrho}&=1\\F_{1}^{\alpha^*{}_\alpha\rho\rho}&=1&F_{1}^{\alpha^*\alphastarrho{}_\alpha\rho}&=-p_1&F_{1}^{\rho1\rho}&=1&F_{1}^{\rho\alpha\alphastarrho}&=1\\F_{1}^{\rho\alpha^*{}_\alpha\rho}&=1&F_{1}^{\rho\rho1}&=1&F_{1}^{\rho\rho\rho}&=1&F_{1}^{\rho\rho{}_\alpha\rho}&=1\\F_{1}^{\rho\rho\alphastarrho}&=1&F_{1}^{\rho{}_\alpha\rho\alpha}&=1&F_{1}^{\rho{}_\alpha\rho\rho}&=1&F_{1}^{\rho{}_\alpha\rho{}_\alpha\rho}&=1\\F_{1}^{\rho{}_\alpha\rho\alphastarrho}&=-1&F_{1}^{\rho\alphastarrho\alpha^*}&=1&F_{1}^{\rho\alphastarrho\rho}&=1&F_{1}^{\rho\alphastarrho{}_\alpha\rho}&=p_1p_2\\F_{1}^{\rho\alphastarrho\alphastarrho}&=-p_1p_2&F_{1}^{{}_\alpha\rho1{}_\alpha\rho}&=1&F_{1}^{{}_\alpha\rho\alpha\rho}&=1&F_{1}^{{}_\alpha\rho\alpha^*\alphastarrho}&=1\\F_{1}^{{}_\alpha\rho\rho\alpha^*}&=1&F_{1}^{{}_\alpha\rho\rho\rho}&=1&F_{1}^{{}_\alpha\rho\rho{}_\alpha\rho}&=1&F_{1}^{{}_\alpha\rho\rho\alphastarrho}&=p_2\\F_{1}^{{}_\alpha\rho{}_\alpha\rho1}&=1&F_{1}^{{}_\alpha\rho{}_\alpha\rho\rho}&=1&F_{1}^{{}_\alpha\rho{}_\alpha\rho{}_\alpha\rho}&=1&F_{1}^{{}_\alpha\rho{}_\alpha\rho\alphastarrho}&=-1\\F_{1}^{{}_\alpha\rho\alphastarrho\alpha}&=-p_1&F_{1}^{{}_\alpha\rho\alphastarrho\rho}&=p_1&F_{1}^{{}_\alpha\rho\alphastarrho{}_\alpha\rho}&=1&F_{1}^{{}_\alpha\rho\alphastarrho\alphastarrho}&=-1\\F_{1}^{\alphastarrho1\alphastarrho}&=1&F_{1}^{\alphastarrho\alpha{}_\alpha\rho}&=1&F_{1}^{\alphastarrho\alpha^*\rho}&=1&F_{1}^{\alphastarrho\rho\alpha}&=1\\F_{1}^{\alphastarrho\rho\rho}&=1&F_{1}^{\alphastarrho\rho{}_\alpha\rho}&=-p_1&F_{1}^{\alphastarrho\rho\alphastarrho}&=1&F_{1}^{\alphastarrho{}_\alpha\rho\alpha^*}&=-p_1\\F_{1}^{\alphastarrho{}_\alpha\rho\rho}&=p_1&F_{1}^{\alphastarrho{}_\alpha\rho{}_\alpha\rho}&=-1&F_{1}^{\alphastarrho{}_\alpha\rho\alphastarrho}&=1&F_{1}^{\alphastarrho\alphastarrho1}&=1\\F_{1}^{\alphastarrho\alphastarrho\rho}&=-p_1p_2&F_{1}^{\alphastarrho\alphastarrho{}_\alpha\rho}&=-1&F_{1}^{\alphastarrho\alphastarrho\alphastarrho}&=1&F_{\alpha}^{11\alpha}&=1\\F_{\alpha}^{1\alpha1}&=1&F_{\alpha}^{1\alpha^*\alpha^*}&=1&F_{\alpha}^{1\rho\alphastarrho}&=1&F_{\alpha}^{1{}_\alpha\rho\rho}&=1\\F_{\alpha}^{1\alphastarrho{}_\alpha\rho}&=1&F_{\alpha}^{\alpha11}&=1&F_{\alpha}^{\alpha\alpha\alpha^*}&=1&F_{\alpha}^{\alpha\alpha^*\alpha}&=1\\F_{\alpha}^{\alpha\rho\rho}&=1&F_{\alpha}^{\alpha{}_\alpha\rho{}_\alpha\rho}&=1&F_{\alpha}^{\alpha\alphastarrho\alphastarrho}&=1&F_{\alpha}^{\alpha^*1\alpha^*}&=1\\F_{\alpha}^{\alpha^*\alpha\alpha}&=1&F_{\alpha}^{\alpha^*\alpha^*1}&=1&F_{\alpha}^{\alpha^*\rho{}_\alpha\rho}&=1&F_{\alpha}^{\alpha^*{}_\alpha\rho\alphastarrho}&=-p_1\\F_{\alpha}^{\alpha^*\alphastarrho\rho}&=-p_1&F_{\alpha}^{\rho1\alphastarrho}&=1&F_{\alpha}^{\rho\alpha{}_\alpha\rho}&=-p_1&F_{\alpha}^{\rho\alpha^*\rho}&=-1\\F_{\alpha}^{\rho\rho\alpha}&=1&F_{\alpha}^{\rho\rho\rho}&=1&F_{\alpha}^{\rho\rho{}_\alpha\rho}&=1&F_{\alpha}^{\rho\rho\alphastarrho}&=1\\F_{\alpha}^{\rho{}_\alpha\rho\alpha^*}&=-p_1&F_{\alpha}^{\rho{}_\alpha\rho\rho}&=p_1&F_{\alpha}^{\rho{}_\alpha\rho{}_\alpha\rho}&=p_1&F_{\alpha}^{\rho{}_\alpha\rho\alphastarrho}&=1\\F_{\alpha}^{\rho\alphastarrho1}&=1&F_{\alpha}^{\rho\alphastarrho\rho}&=-p_1p_2&F_{\alpha}^{\rho\alphastarrho{}_\alpha\rho}&=p_1&F_{\alpha}^{\rho\alphastarrho\alphastarrho}&=1\\F_{\alpha}^{{}_\alpha\rho1\rho}&=1&F_{\alpha}^{{}_\alpha\rho\alpha\alphastarrho}&=-1&F_{\alpha}^{{}_\alpha\rho\alpha^*{}_\alpha\rho}&=-1&F_{\alpha}^{{}_\alpha\rho\rho1}&=1\\F_{\alpha}^{{}_\alpha\rho\rho\rho}&=1&F_{\alpha}^{{}_\alpha\rho\rho{}_\alpha\rho}&=1&F_{\alpha}^{{}_\alpha\rho\rho\alphastarrho}&=-1&F_{\alpha}^{{}_\alpha\rho{}_\alpha\rho\alpha}&=1\\F_{\alpha}^{{}_\alpha\rho{}_\alpha\rho\rho}&=1&F_{\alpha}^{{}_\alpha\rho{}_\alpha\rho{}_\alpha\rho}&=1&F_{\alpha}^{{}_\alpha\rho{}_\alpha\rho\alphastarrho}&=-1&F_{\alpha}^{{}_\alpha\rho\alphastarrho\alpha^*}&=-p_1\\F_{\alpha}^{{}_\alpha\rho\alphastarrho\rho}&=1&F_{\alpha}^{{}_\alpha\rho\alphastarrho{}_\alpha\rho}&=1&F_{\alpha}^{{}_\alpha\rho\alphastarrho\alphastarrho}&=-p_1p_2&F_{\alpha}^{\alphastarrho1{}_\alpha\rho}&=1\\F_{\alpha}^{\alphastarrho\alpha\rho}&=p_1&F_{\alpha}^{\alphastarrho\alpha^*\alphastarrho}&=1&F_{\alpha}^{\alphastarrho\rho\alpha^*}&=1&F_{\alpha}^{\alphastarrho\rho\rho}&=p_1\\F_{\alpha}^{\alphastarrho\rho{}_\alpha\rho}&=1&F_{\alpha}^{\alphastarrho\rho\alphastarrho}&=p_2&F_{\alpha}^{\alphastarrho{}_\alpha\rho1}&=1&F_{\alpha}^{\alphastarrho{}_\alpha\rho\rho}&=-p_1\\F_{\alpha}^{\alphastarrho{}_\alpha\rho{}_\alpha\rho}&=1&F_{\alpha}^{\alphastarrho{}_\alpha\rho\alphastarrho}&=-1&F_{\alpha}^{\alphastarrho\alphastarrho\alpha}&=1&F_{\alpha}^{\alphastarrho\alphastarrho\rho}&=-1\\F_{\alpha}^{\alphastarrho\alphastarrho{}_\alpha\rho}&=1&F_{\alpha}^{\alphastarrho\alphastarrho\alphastarrho}&=-p_1p_2&F_{\alpha^*}^{11\alpha^*}&=1&F_{\alpha^*}^{1\alpha\alpha}&=1\\F_{\alpha^*}^{1\alpha^*1}&=1&F_{\alpha^*}^{1\rho{}_\alpha\rho}&=1&F_{\alpha^*}^{1{}_\alpha\rho\alphastarrho}&=1&F_{\alpha^*}^{1\alphastarrho\rho}&=1\\F_{\alpha^*}^{\alpha1\alpha}&=1&F_{\alpha^*}^{\alpha\alpha1}&=1&F_{\alpha^*}^{\alpha\alpha^*\alpha^*}&=1&F_{\alpha^*}^{\alpha\rho\alphastarrho}&=-p_1\\F_{\alpha^*}^{\alpha{}_\alpha\rho\rho}&=1&F_{\alpha^*}^{\alpha\alphastarrho{}_\alpha\rho}&=1&F_{\alpha^*}^{\alpha^*11}&=1&F_{\alpha^*}^{\alpha^*\alpha\alpha^*}&=1\\F_{\alpha^*}^{\alpha^*\alpha^*\alpha}&=1&F_{\alpha^*}^{\alpha^*\rho\rho}&=1&F_{\alpha^*}^{\alpha^*{}_\alpha\rho{}_\alpha\rho}&=1&F_{\alpha^*}^{\alpha^*\alphastarrho\alphastarrho}&=1\\F_{\alpha^*}^{\rho1{}_\alpha\rho}&=1&F_{\alpha^*}^{\rho\alpha\rho}&=-1&F_{\alpha^*}^{\rho\alpha^*\alphastarrho}&=p_1&F_{\alpha^*}^{\rho\rho\alpha^*}&=1\\F_{\alpha^*}^{\rho\rho\rho}&=-1&F_{\alpha^*}^{\rho\rho{}_\alpha\rho}&=1&F_{\alpha^*}^{\rho\rho\alphastarrho}&=-p_1p_2&F_{\alpha^*}^{\rho{}_\alpha\rho1}&=1\\F_{\alpha^*}^{\rho{}_\alpha\rho\rho}&=1&F_{\alpha^*}^{\rho{}_\alpha\rho{}_\alpha\rho}&=1&F_{\alpha^*}^{\rho{}_\alpha\rho\alphastarrho}&=p_1&F_{\alpha^*}^{\rho\alphastarrho\alpha}&=1\\F_{\alpha^*}^{\rho\alphastarrho\rho}&=p_1&F_{\alpha^*}^{\rho\alphastarrho{}_\alpha\rho}&=1&F_{\alpha^*}^{\rho\alphastarrho\alphastarrho}&=p_2&F_{\alpha^*}^{{}_\alpha\rho1\alphastarrho}&=1\\F_{\alpha^*}^{{}_\alpha\rho\alpha{}_\alpha\rho}&=-1&F_{\alpha^*}^{{}_\alpha\rho\alpha^*\rho}&=-p_1&F_{\alpha^*}^{{}_\alpha\rho\rho\alpha}&=-p_1&F_{\alpha^*}^{{}_\alpha\rho\rho\rho}&=-p_1\\F_{\alpha^*}^{{}_\alpha\rho\rho{}_\alpha\rho}&=p_1&F_{\alpha^*}^{{}_\alpha\rho\rho\alphastarrho}&=1&F_{\alpha^*}^{{}_\alpha\rho{}_\alpha\rho\alpha^*}&=1&F_{\alpha^*}^{{}_\alpha\rho{}_\alpha\rho\rho}&=-1\\F_{\alpha^*}^{{}_\alpha\rho{}_\alpha\rho{}_\alpha\rho}&=-1&F_{\alpha^*}^{{}_\alpha\rho{}_\alpha\rho\alphastarrho}&=1&F_{\alpha^*}^{{}_\alpha\rho\alphastarrho1}&=1&F_{\alpha^*}^{{}_\alpha\rho\alphastarrho\rho}&=p_1\\F_{\alpha^*}^{{}_\alpha\rho\alphastarrho{}_\alpha\rho}&=-1&F_{\alpha^*}^{{}_\alpha\rho\alphastarrho\alphastarrho}&=1&F_{\alpha^*}^{\alphastarrho1\rho}&=1&F_{\alpha^*}^{\alphastarrho\alpha\alphastarrho}&=1\\F_{\alpha^*}^{\alphastarrho\alpha^*{}_\alpha\rho}&=-1&F_{\alpha^*}^{\alphastarrho\rho1}&=1&F_{\alpha^*}^{\alphastarrho\rho\rho}&=1&F_{\alpha^*}^{\alphastarrho\rho{}_\alpha\rho}&=1\\F_{\alpha^*}^{\alphastarrho\rho\alphastarrho}&=1&F_{\alpha^*}^{\alphastarrho{}_\alpha\rho\alpha}&=1&F_{\alpha^*}^{\alphastarrho{}_\alpha\rho\rho}&=1&F_{\alpha^*}^{\alphastarrho{}_\alpha\rho{}_\alpha\rho}&=1\\F_{\alpha^*}^{\alphastarrho{}_\alpha\rho\alphastarrho}&=-1&F_{\alpha^*}^{\alphastarrho\alphastarrho\alpha^*}&=1&F_{\alpha^*}^{\alphastarrho\alphastarrho\rho}&=1&F_{\alpha^*}^{\alphastarrho\alphastarrho{}_\alpha\rho}&=p_1p_2\\F_{\alpha^*}^{\alphastarrho\alphastarrho\alphastarrho}&=-p_1p_2&F_{\rho}^{11\rho}&=1&F_{\rho}^{1\alpha\alphastarrho}&=1&F_{\rho}^{1\alpha^*{}_\alpha\rho}&=1\\F_{\rho}^{1\rho1}&=1&F_{\rho}^{1\rho\rho}&=1&F_{\rho}^{1\rho{}_\alpha\rho}&=1&F_{\rho}^{1\rho\alphastarrho}&=1\\F_{\rho}^{1{}_\alpha\rho\alpha}&=1&F_{\rho}^{1{}_\alpha\rho\rho}&=1&F_{\rho}^{1{}_\alpha\rho{}_\alpha\rho}&=1&F_{\rho}^{1{}_\alpha\rho\alphastarrho}&=1\\F_{\rho}^{1\alphastarrho\alpha^*}&=1&F_{\rho}^{1\alphastarrho\rho}&=1&F_{\rho}^{1\alphastarrho{}_\alpha\rho}&=1&F_{\rho}^{1\alphastarrho\alphastarrho}&=1\\F_{\rho}^{\alpha1\alphastarrho}&=1&F_{\rho}^{\alpha\alpha{}_\alpha\rho}&=1&F_{\rho}^{\alpha\alpha^*\rho}&=1&F_{\rho}^{\alpha\rho\alpha}&=-1\\F_{\rho}^{\alpha\rho\rho}&=-1&F_{\rho}^{\alpha\rho{}_\alpha\rho}&=1&F_{\rho}^{\alpha\rho\alphastarrho}&=1&F_{\rho}^{\alpha{}_\alpha\rho\alpha^*}&=1\\F_{\rho}^{\alpha{}_\alpha\rho\rho}&=1&F_{\rho}^{\alpha{}_\alpha\rho{}_\alpha\rho}&=1&F_{\rho}^{\alpha{}_\alpha\rho\alphastarrho}&=p_1p_2&F_{\rho}^{\alpha\alphastarrho1}&=1\\F_{\rho}^{\alpha\alphastarrho\rho}&=1&F_{\rho}^{\alpha\alphastarrho{}_\alpha\rho}&=1&F_{\rho}^{\alpha\alphastarrho\alphastarrho}&=1&F_{\rho}^{\alpha^*1{}_\alpha\rho}&=1\\F_{\rho}^{\alpha^*\alpha\rho}&=1&F_{\rho}^{\alpha^*\alpha^*\alphastarrho}&=-p_1&F_{\rho}^{\alpha^*\rho\alpha^*}&=-1&F_{\rho}^{\alpha^*\rho\rho}&=1\\F_{\rho}^{\alpha^*\rho{}_\alpha\rho}&=1&F_{\rho}^{\alpha^*\rho\alphastarrho}&=1&F_{\rho}^{\alpha^*{}_\alpha\rho1}&=1&F_{\rho}^{\alpha^*{}_\alpha\rho\rho}&=1\\F_{\rho}^{\alpha^*{}_\alpha\rho{}_\alpha\rho}&=1&F_{\rho}^{\alpha^*{}_\alpha\rho\alphastarrho}&=1&F_{\rho}^{\alpha^*\alphastarrho\alpha}&=-1&F_{\rho}^{\alpha^*\alphastarrho\rho}&=-1\\F_{\rho}^{\alpha^*\alphastarrho{}_\alpha\rho}&=1&F_{\rho}^{\alpha^*\alphastarrho\alphastarrho}&=1&F_{\rho}^{\rho11}&=1&F_{\rho}^{\rho1\rho}&=1\\F_{\rho}^{\rho1{}_\alpha\rho}&=1&F_{\rho}^{\rho1\alphastarrho}&=1&F_{\rho}^{\rho\alpha\alpha^*}&=1&F_{\rho}^{\rho\alpha\rho}&=-1\\F_{\rho}^{\rho\alpha{}_\alpha\rho}&=p_1&F_{\rho}^{\rho\alpha\alphastarrho}&=-p_1p_2&F_{\rho}^{\rho\alpha^*\alpha}&=1&F_{\rho}^{\rho\alpha^*\rho}&=1\\F_{\rho}^{\rho\alpha^*{}_\alpha\rho}&=1&F_{\rho}^{\rho\alpha^*\alphastarrho}&=-p_1&F_{\rho}^{\rho\rho1}&=1&F_{\rho}^{\rho\rho\alpha}&=-1\\F_{\rho}^{\rho\rho\alpha^*}&=1&F_{\rho}^{\rho{}_\alpha\rho1}&=1&F_{\rho}^{\rho{}_\alpha\rho\alpha}&=1&F_{\rho}^{\rho{}_\alpha\rho\alpha^*}&=-p_1\\F_{\rho}^{\rho\alphastarrho1}&=1&F_{\rho}^{\rho\alphastarrho\alpha}&=p_2&F_{\rho}^{\rho\alphastarrho\alpha^*}&=1&F_{\rho}^{{}_\alpha\rho1\alpha}&=1\\F_{\rho}^{{}_\alpha\rho1\rho}&=1&F_{\rho}^{{}_\alpha\rho1{}_\alpha\rho}&=1&F_{\rho}^{{}_\alpha\rho1\alphastarrho}&=1&F_{\rho}^{{}_\alpha\rho\alpha1}&=1\\F_{\rho}^{{}_\alpha\rho\alpha\rho}&=1&F_{\rho}^{{}_\alpha\rho\alpha{}_\alpha\rho}&=1&F_{\rho}^{{}_\alpha\rho\alpha\alphastarrho}&=1&F_{\rho}^{{}_\alpha\rho\alpha^*\alpha^*}&=-p_1\\F_{\rho}^{{}_\alpha\rho\alpha^*\rho}&=p_1&F_{\rho}^{{}_\alpha\rho\alpha^*{}_\alpha\rho}&=-1&F_{\rho}^{{}_\alpha\rho\alpha^*\alphastarrho}&=-p_1p_2&F_{\rho}^{{}_\alpha\rho\rho1}&=1\\F_{\rho}^{{}_\alpha\rho\rho\alpha}&=-1&F_{\rho}^{{}_\alpha\rho\rho\alpha^*}&=p_1&F_{\rho}^{{}_\alpha\rho{}_\alpha\rho1}&=1&F_{\rho}^{{}_\alpha\rho{}_\alpha\rho\alpha}&=1\\F_{\rho}^{{}_\alpha\rho{}_\alpha\rho\alpha^*}&=1&F_{\rho}^{{}_\alpha\rho\alphastarrho1}&=1&F_{\rho}^{{}_\alpha\rho\alphastarrho\alpha}&=-1&F_{\rho}^{{}_\alpha\rho\alphastarrho\alpha^*}&=1\\F_{\rho}^{\alphastarrho1\alpha^*}&=1&F_{\rho}^{\alphastarrho1\rho}&=1&F_{\rho}^{\alphastarrho1{}_\alpha\rho}&=1&F_{\rho}^{\alphastarrho1\alphastarrho}&=1\\F_{\rho}^{\alphastarrho\alpha\alpha}&=1&F_{\rho}^{\alphastarrho\alpha\rho}&=p_1&F_{\rho}^{\alphastarrho\alpha{}_\alpha\rho}&=-p_1p_2&F_{\rho}^{\alphastarrho\alpha\alphastarrho}&=-1\\F_{\rho}^{\alphastarrho\alpha^*1}&=1&F_{\rho}^{\alphastarrho\alpha^*\rho}&=-p_1p_2&F_{\rho}^{\alphastarrho\alpha^*{}_\alpha\rho}&=-1&F_{\rho}^{\alphastarrho\alpha^*\alphastarrho}&=-1\\F_{\rho}^{\alphastarrho\rho1}&=1&F_{\rho}^{\alphastarrho\rho\alpha}&=-p_2&F_{\rho}^{\alphastarrho\rho\alpha^*}&=-1&F_{\rho}^{\alphastarrho{}_\alpha\rho1}&=1\\F_{\rho}^{\alphastarrho{}_\alpha\rho\alpha}&=1&F_{\rho}^{\alphastarrho{}_\alpha\rho\alpha^*}&=p_1p_2&F_{\rho}^{\alphastarrho\alphastarrho1}&=1&F_{\rho}^{\alphastarrho\alphastarrho\alpha}&=-1\\F_{\rho}^{\alphastarrho\alphastarrho\alpha^*}&=1&F_{{}_\alpha\rho}^{11{}_\alpha\rho}&=1&F_{{}_\alpha\rho}^{1\alpha\rho}&=1&F_{{}_\alpha\rho}^{1\alpha^*\alphastarrho}&=1\\F_{{}_\alpha\rho}^{1\rho\alpha^*}&=1&F_{{}_\alpha\rho}^{1\rho\rho}&=1&F_{{}_\alpha\rho}^{1\rho{}_\alpha\rho}&=1&F_{{}_\alpha\rho}^{1\rho\alphastarrho}&=1\\F_{{}_\alpha\rho}^{1{}_\alpha\rho1}&=1&F_{{}_\alpha\rho}^{1{}_\alpha\rho\rho}&=1&F_{{}_\alpha\rho}^{1{}_\alpha\rho{}_\alpha\rho}&=1&F_{{}_\alpha\rho}^{1{}_\alpha\rho\alphastarrho}&=1\\F_{{}_\alpha\rho}^{1\alphastarrho\alpha}&=1&F_{{}_\alpha\rho}^{1\alphastarrho\rho}&=1&F_{{}_\alpha\rho}^{1\alphastarrho{}_\alpha\rho}&=1&F_{{}_\alpha\rho}^{1\alphastarrho\alphastarrho}&=1\\F_{{}_\alpha\rho}^{\alpha1\rho}&=1&F_{{}_\alpha\rho}^{\alpha\alpha\alphastarrho}&=-p_1&F_{{}_\alpha\rho}^{\alpha\alpha^*{}_\alpha\rho}&=1&F_{{}_\alpha\rho}^{\alpha\rho1}&=1\\F_{{}_\alpha\rho}^{\alpha\rho\rho}&=1&F_{{}_\alpha\rho}^{\alpha\rho{}_\alpha\rho}&=1&F_{{}_\alpha\rho}^{\alpha\rho\alphastarrho}&=1&F_{{}_\alpha\rho}^{\alpha{}_\alpha\rho\alpha}&=p_1\\F_{{}_\alpha\rho}^{\alpha{}_\alpha\rho\rho}&=p_1&F_{{}_\alpha\rho}^{\alpha{}_\alpha\rho{}_\alpha\rho}&=-p_1&F_{{}_\alpha\rho}^{\alpha{}_\alpha\rho\alphastarrho}&=-p_1&F_{{}_\alpha\rho}^{\alpha\alphastarrho\alpha^*}&=-1\\F_{{}_\alpha\rho}^{\alpha\alphastarrho\rho}&=1&F_{{}_\alpha\rho}^{\alpha\alphastarrho{}_\alpha\rho}&=1&F_{{}_\alpha\rho}^{\alpha\alphastarrho\alphastarrho}&=1&F_{{}_\alpha\rho}^{\alpha^*1\alphastarrho}&=1\\F_{{}_\alpha\rho}^{\alpha^*\alpha{}_\alpha\rho}&=-p_1&F_{{}_\alpha\rho}^{\alpha^*\alpha^*\rho}&=-p_1&F_{{}_\alpha\rho}^{\alpha^*\rho\alpha}&=1&F_{{}_\alpha\rho}^{\alpha^*\rho\rho}&=1\\F_{{}_\alpha\rho}^{\alpha^*\rho{}_\alpha\rho}&=1&F_{{}_\alpha\rho}^{\alpha^*\rho\alphastarrho}&=1&F_{{}_\alpha\rho}^{\alpha^*{}_\alpha\rho\alpha^*}&=p_1&F_{{}_\alpha\rho}^{\alpha^*{}_\alpha\rho\rho}&=-p_1\\F_{{}_\alpha\rho}^{\alpha^*{}_\alpha\rho{}_\alpha\rho}&=-p_1&F_{{}_\alpha\rho}^{\alpha^*{}_\alpha\rho\alphastarrho}&=-p_2&F_{{}_\alpha\rho}^{\alpha^*\alphastarrho1}&=1&F_{{}_\alpha\rho}^{\alpha^*\alphastarrho\rho}&=1\\F_{{}_\alpha\rho}^{\alpha^*\alphastarrho{}_\alpha\rho}&=1&F_{{}_\alpha\rho}^{\alpha^*\alphastarrho\alphastarrho}&=1&F_{{}_\alpha\rho}^{\rho1\alpha^*}&=1&F_{{}_\alpha\rho}^{\rho1\rho}&=1\\F_{{}_\alpha\rho}^{\rho1{}_\alpha\rho}&=1&F_{{}_\alpha\rho}^{\rho1\alphastarrho}&=1&F_{{}_\alpha\rho}^{\rho\alpha\alpha}&=1&F_{{}_\alpha\rho}^{\rho\alpha\rho}&=-p_1\\F_{{}_\alpha\rho}^{\rho\alpha{}_\alpha\rho}&=-p_1p_2&F_{{}_\alpha\rho}^{\rho\alpha\alphastarrho}&=-1&F_{{}_\alpha\rho}^{\rho\alpha^*1}&=1&F_{{}_\alpha\rho}^{\rho\alpha^*\rho}&=p_1p_2\\F_{{}_\alpha\rho}^{\rho\alpha^*{}_\alpha\rho}&=1&F_{{}_\alpha\rho}^{\rho\alpha^*\alphastarrho}&=1&F_{{}_\alpha\rho}^{\rho\rho1}&=1&F_{{}_\alpha\rho}^{\rho\rho\alpha}&=p_2\\F_{{}_\alpha\rho}^{\rho\rho\alpha^*}&=1&F_{{}_\alpha\rho}^{\rho{}_\alpha\rho1}&=1&F_{{}_\alpha\rho}^{\rho{}_\alpha\rho\alpha}&=-1&F_{{}_\alpha\rho}^{\rho{}_\alpha\rho\alpha^*}&=-p_1p_2\\F_{{}_\alpha\rho}^{\rho\alphastarrho1}&=1&F_{{}_\alpha\rho}^{\rho\alphastarrho\alpha}&=1&F_{{}_\alpha\rho}^{\rho\alphastarrho\alpha^*}&=-1&F_{{}_\alpha\rho}^{{}_\alpha\rho11}&=1\\F_{{}_\alpha\rho}^{{}_\alpha\rho1\rho}&=1&F_{{}_\alpha\rho}^{{}_\alpha\rho1{}_\alpha\rho}&=1&F_{{}_\alpha\rho}^{{}_\alpha\rho1\alphastarrho}&=1&F_{{}_\alpha\rho}^{{}_\alpha\rho\alpha\alpha^*}&=1\\F_{{}_\alpha\rho}^{{}_\alpha\rho\alpha\rho}&=1&F_{{}_\alpha\rho}^{{}_\alpha\rho\alpha{}_\alpha\rho}&=-p_1&F_{{}_\alpha\rho}^{{}_\alpha\rho\alpha\alphastarrho}&=p_1p_2&F_{{}_\alpha\rho}^{{}_\alpha\rho\alpha^*\alpha}&=-p_1\\F_{{}_\alpha\rho}^{{}_\alpha\rho\alpha^*\rho}&=-p_1&F_{{}_\alpha\rho}^{{}_\alpha\rho\alpha^*{}_\alpha\rho}&=p_1&F_{{}_\alpha\rho}^{{}_\alpha\rho\alpha^*\alphastarrho}&=-1&F_{{}_\alpha\rho}^{{}_\alpha\rho\rho1}&=1\\F_{{}_\alpha\rho}^{{}_\alpha\rho\rho\alpha}&=-p_1&F_{{}_\alpha\rho}^{{}_\alpha\rho\rho\alpha^*}&=1&F_{{}_\alpha\rho}^{{}_\alpha\rho{}_\alpha\rho1}&=1&F_{{}_\alpha\rho}^{{}_\alpha\rho{}_\alpha\rho\alpha}&=p_1\\F_{{}_\alpha\rho}^{{}_\alpha\rho{}_\alpha\rho\alpha^*}&=p_1&F_{{}_\alpha\rho}^{{}_\alpha\rho\alphastarrho1}&=1&F_{{}_\alpha\rho}^{{}_\alpha\rho\alphastarrho\alpha}&=1&F_{{}_\alpha\rho}^{{}_\alpha\rho\alphastarrho\alpha^*}&=-1\\F_{{}_\alpha\rho}^{\alphastarrho1\alpha}&=1&F_{{}_\alpha\rho}^{\alphastarrho1\rho}&=1&F_{{}_\alpha\rho}^{\alphastarrho1{}_\alpha\rho}&=1&F_{{}_\alpha\rho}^{\alphastarrho1\alphastarrho}&=1\\F_{{}_\alpha\rho}^{\alphastarrho\alpha1}&=1&F_{{}_\alpha\rho}^{\alphastarrho\alpha\rho}&=-1&F_{{}_\alpha\rho}^{\alphastarrho\alpha{}_\alpha\rho}&=-1&F_{{}_\alpha\rho}^{\alphastarrho\alpha\alphastarrho}&=1\\F_{{}_\alpha\rho}^{\alphastarrho\alpha^*\alpha^*}&=1&F_{{}_\alpha\rho}^{\alphastarrho\alpha^*\rho}&=-1&F_{{}_\alpha\rho}^{\alphastarrho\alpha^*{}_\alpha\rho}&=-p_1&F_{{}_\alpha\rho}^{\alphastarrho\alpha^*\alphastarrho}&=p_2\\F_{{}_\alpha\rho}^{\alphastarrho\rho1}&=1&F_{{}_\alpha\rho}^{\alphastarrho\rho\alpha}&=1&F_{{}_\alpha\rho}^{\alphastarrho\rho\alpha^*}&=1&F_{{}_\alpha\rho}^{\alphastarrho{}_\alpha\rho1}&=1\\F_{{}_\alpha\rho}^{\alphastarrho{}_\alpha\rho\alpha}&=1&F_{{}_\alpha\rho}^{\alphastarrho{}_\alpha\rho\alpha^*}&=p_1&F_{{}_\alpha\rho}^{\alphastarrho\alphastarrho1}&=1&F_{{}_\alpha\rho}^{\alphastarrho\alphastarrho\alpha}&=1\\F_{{}_\alpha\rho}^{\alphastarrho\alphastarrho\alpha^*}&=-p_2&F_{\alphastarrho}^{11\alphastarrho}&=1&F_{\alphastarrho}^{1\alpha{}_\alpha\rho}&=1&F_{\alphastarrho}^{1\alpha^*\rho}&=1\\F_{\alphastarrho}^{1\rho\alpha}&=1&F_{\alphastarrho}^{1\rho\rho}&=1&F_{\alphastarrho}^{1\rho{}_\alpha\rho}&=1&F_{\alphastarrho}^{1\rho\alphastarrho}&=1\\F_{\alphastarrho}^{1{}_\alpha\rho\alpha^*}&=1&F_{\alphastarrho}^{1{}_\alpha\rho\rho}&=1&F_{\alphastarrho}^{1{}_\alpha\rho{}_\alpha\rho}&=1&F_{\alphastarrho}^{1{}_\alpha\rho\alphastarrho}&=1\\F_{\alphastarrho}^{1\alphastarrho1}&=1&F_{\alphastarrho}^{1\alphastarrho\rho}&=1&F_{\alphastarrho}^{1\alphastarrho{}_\alpha\rho}&=1&F_{\alphastarrho}^{1\alphastarrho\alphastarrho}&=1\\F_{\alphastarrho}^{\alpha1{}_\alpha\rho}&=1&F_{\alphastarrho}^{\alpha\alpha\rho}&=1&F_{\alphastarrho}^{\alpha\alpha^*\alphastarrho}&=-p_1&F_{\alphastarrho}^{\alpha\rho\alpha^*}&=-1\\F_{\alphastarrho}^{\alpha\rho\rho}&=1&F_{\alphastarrho}^{\alpha\rho{}_\alpha\rho}&=1&F_{\alphastarrho}^{\alpha\rho\alphastarrho}&=p_1p_2&F_{\alphastarrho}^{\alpha{}_\alpha\rho1}&=1\\F_{\alphastarrho}^{\alpha{}_\alpha\rho\rho}&=1&F_{\alphastarrho}^{\alpha{}_\alpha\rho{}_\alpha\rho}&=1&F_{\alphastarrho}^{\alpha{}_\alpha\rho\alphastarrho}&=1&F_{\alphastarrho}^{\alpha\alphastarrho\alpha}&=-p_1\\F_{\alphastarrho}^{\alpha\alphastarrho\rho}&=-p_1&F_{\alphastarrho}^{\alpha\alphastarrho{}_\alpha\rho}&=-p_1&F_{\alphastarrho}^{\alpha\alphastarrho\alphastarrho}&=-p_1&F_{\alphastarrho}^{\alpha^*1\rho}&=1\\F_{\alphastarrho}^{\alpha^*\alpha\alphastarrho}&=1&F_{\alphastarrho}^{\alpha^*\alpha^*{}_\alpha\rho}&=1&F_{\alphastarrho}^{\alpha^*\rho1}&=1&F_{\alphastarrho}^{\alpha^*\rho\rho}&=1\\F_{\alphastarrho}^{\alpha^*\rho{}_\alpha\rho}&=1&F_{\alphastarrho}^{\alpha^*\rho\alphastarrho}&=1&F_{\alphastarrho}^{\alpha^*{}_\alpha\rho\alpha}&=-1&F_{\alphastarrho}^{\alpha^*{}_\alpha\rho\rho}&=-1\\F_{\alphastarrho}^{\alpha^*{}_\alpha\rho{}_\alpha\rho}&=1&F_{\alphastarrho}^{\alpha^*{}_\alpha\rho\alphastarrho}&=1&F_{\alphastarrho}^{\alpha^*\alphastarrho\alpha^*}&=-p_1&F_{\alphastarrho}^{\alpha^*\alphastarrho\rho}&=-p_1\\F_{\alphastarrho}^{\alpha^*\alphastarrho{}_\alpha\rho}&=-p_1&F_{\alphastarrho}^{\alpha^*\alphastarrho\alphastarrho}&=-p_2&F_{\alphastarrho}^{\rho1\alpha}&=1&F_{\alphastarrho}^{\rho1\rho}&=1\\F_{\alphastarrho}^{\rho1{}_\alpha\rho}&=1&F_{\alphastarrho}^{\rho1\alphastarrho}&=1&F_{\alphastarrho}^{\rho\alpha1}&=1&F_{\alphastarrho}^{\rho\alpha\rho}&=-1\\F_{\alphastarrho}^{\rho\alpha{}_\alpha\rho}&=-1&F_{\alphastarrho}^{\rho\alpha\alphastarrho}&=1&F_{\alphastarrho}^{\rho\alpha^*\alpha^*}&=-p_1&F_{\alphastarrho}^{\rho\alpha^*\rho}&=-p_1\\F_{\alphastarrho}^{\rho\alpha^*{}_\alpha\rho}&=-1&F_{\alphastarrho}^{\rho\alpha^*\alphastarrho}&=1&F_{\alphastarrho}^{\rho\rho1}&=1&F_{\alphastarrho}^{\rho\rho\alpha}&=1\\F_{\alphastarrho}^{\rho\rho\alpha^*}&=p_1&F_{\alphastarrho}^{\rho{}_\alpha\rho1}&=1&F_{\alphastarrho}^{\rho{}_\alpha\rho\alpha}&=1&F_{\alphastarrho}^{\rho{}_\alpha\rho\alpha^*}&=1\\F_{\alphastarrho}^{\rho\alphastarrho1}&=1&F_{\alphastarrho}^{\rho\alphastarrho\alpha}&=1&F_{\alphastarrho}^{\rho\alphastarrho\alpha^*}&=-p_1p_2&F_{\alphastarrho}^{{}_\alpha\rho1\alpha^*}&=1\\F_{\alphastarrho}^{{}_\alpha\rho1\rho}&=1&F_{\alphastarrho}^{{}_\alpha\rho1{}_\alpha\rho}&=1&F_{\alphastarrho}^{{}_\alpha\rho1\alphastarrho}&=1&F_{\alphastarrho}^{{}_\alpha\rho\alpha\alpha}&=-p_1\\F_{\alphastarrho}^{{}_\alpha\rho\alpha\rho}&=-1&F_{\alphastarrho}^{{}_\alpha\rho\alpha{}_\alpha\rho}&=-p_1&F_{\alphastarrho}^{{}_\alpha\rho\alpha\alphastarrho}&=-p_1&F_{\alphastarrho}^{{}_\alpha\rho\alpha^*1}&=1\\F_{\alphastarrho}^{{}_\alpha\rho\alpha^*\rho}&=-1&F_{\alphastarrho}^{{}_\alpha\rho\alpha^*{}_\alpha\rho}&=-1&F_{\alphastarrho}^{{}_\alpha\rho\alpha^*\alphastarrho}&=-1&F_{\alphastarrho}^{{}_\alpha\rho\rho1}&=1\\F_{\alphastarrho}^{{}_\alpha\rho\rho\alpha}&=1&F_{\alphastarrho}^{{}_\alpha\rho\rho\alpha^*}&=-1&F_{\alphastarrho}^{{}_\alpha\rho{}_\alpha\rho1}&=1&F_{\alphastarrho}^{{}_\alpha\rho{}_\alpha\rho\alpha}&=p_1\\F_{\alphastarrho}^{{}_\alpha\rho{}_\alpha\rho\alpha^*}&=1&F_{\alphastarrho}^{{}_\alpha\rho\alphastarrho1}&=1&F_{\alphastarrho}^{{}_\alpha\rho\alphastarrho\alpha}&=-p_1&F_{\alphastarrho}^{{}_\alpha\rho\alphastarrho\alpha^*}&=1\\F_{\alphastarrho}^{\alphastarrho11}&=1&F_{\alphastarrho}^{\alphastarrho1\rho}&=1&F_{\alphastarrho}^{\alphastarrho1{}_\alpha\rho}&=1&F_{\alphastarrho}^{\alphastarrho1\alphastarrho}&=1\\F_{\alphastarrho}^{\alphastarrho\alpha\alpha^*}&=-p_1&F_{\alphastarrho}^{\alphastarrho\alpha\rho}&=p_1&F_{\alphastarrho}^{\alphastarrho\alpha{}_\alpha\rho}&=-1&F_{\alphastarrho}^{\alphastarrho\alpha\alphastarrho}&=p_1\\F_{\alphastarrho}^{\alphastarrho\alpha^*\alpha}&=1&F_{\alphastarrho}^{\alphastarrho\alpha^*\rho}&=1&F_{\alphastarrho}^{\alphastarrho\alpha^*{}_\alpha\rho}&=-1&F_{\alphastarrho}^{\alphastarrho\alpha^*\alphastarrho}&=p_1\\F_{\alphastarrho}^{\alphastarrho\rho1}&=1&F_{\alphastarrho}^{\alphastarrho\rho\alpha}&=1&F_{\alphastarrho}^{\alphastarrho\rho\alpha^*}&=-p_1&F_{\alphastarrho}^{\alphastarrho{}_\alpha\rho1}&=1\\F_{\alphastarrho}^{\alphastarrho{}_\alpha\rho\alpha}&=-1&F_{\alphastarrho}^{\alphastarrho{}_\alpha\rho\alpha^*}&=1&F_{\alphastarrho}^{\alphastarrho\alphastarrho1}&=1&F_{\alphastarrho}^{\alphastarrho\alphastarrho\alpha}&=-p_2\\F_{\alphastarrho}^{\alphastarrho\alphastarrho\alpha^*}&=-p_1
\end{align}}
\subsection{$3$-dimensional $F$-symbols}
{\allowdisplaybreaks\begin{align}
F_{\rho}^{\rho\rho{}_\alpha\rho}=&{\small\begin{array}{c|ccc}&\rho&{}_\alpha\rho&\alphastarrho\\\hline\rho&D_+&D_-&-C\\{}_\alpha\rho&D_-&C&-D_+\\\alphastarrho&C&D_+&-D_-\end{array}}&F_{\rho}^{\rho\rho\alphastarrho}=&{\small\begin{array}{c|ccc}&\rho&{}_\alpha\rho&\alphastarrho\\\hline\rho&D_-&C&-D_+p_1p_2\\{}_\alpha\rho&Cp_1p_2&D_+p_1p_2&-D_-\\\alphastarrho&D_+&D_-&-Cp_1p_2\end{array}}\\[0.5\baselineskip]F_{\rho}^{\rho{}_\alpha\rho\rho}=&{\small\begin{array}{c|ccc}&\rho&{}_\alpha\rho&\alphastarrho\\\hline\rho&D_+&D_-&C\\{}_\alpha\rho&D_-&C&D_+\\\alphastarrho&-Cp_1&-D_+p_1&-D_-p_1\end{array}}&F_{\rho}^{\rho{}_\alpha\rho\alphastarrho}=&{\small\begin{array}{c|ccc}&\rho&{}_\alpha\rho&\alphastarrho\\\hline\rho&-Cp_1&D_+&D_-p_1p_2\\{}_\alpha\rho&D_+&-D_-p_1&-Cp_2\\\alphastarrho&D_-&-Cp_1&-D_+p_2\end{array}}\\[0.5\baselineskip]F_{\rho}^{\rho\alphastarrho\rho}=&{\small\begin{array}{c|ccc}&\rho&{}_\alpha\rho&\alphastarrho\\\hline\rho&D_-&Cp_1p_2&D_+\\{}_\alpha\rho&Cp_1&D_+p_2&D_-p_1\\\alphastarrho&D_+&D_-p_1p_2&C\end{array}}&F_{\rho}^{\rho\alphastarrho{}_\alpha\rho}=&{\small\begin{array}{c|ccc}&\rho&{}_\alpha\rho&\alphastarrho\\\hline\rho&-Cp_1&D_+p_1p_2&D_-\\{}_\alpha\rho&D_+&-D_-p_2&-Cp_1\\\alphastarrho&D_-&-Cp_2&-D_+p_1\end{array}}\\[0.5\baselineskip]F_{\rho}^{{}_\alpha\rho\rho\rho}=&{\small\begin{array}{c|ccc}&\rho&{}_\alpha\rho&\alphastarrho\\\hline\rho&D_+&D_-&-Cp_1\\{}_\alpha\rho&D_-&C&-D_+p_1\\\alphastarrho&-C&-D_+&D_-p_1\end{array}}&F_{\rho}^{{}_\alpha\rho\rho{}_\alpha\rho}=&{\small\begin{array}{c|ccc}&\rho&{}_\alpha\rho&\alphastarrho\\\hline\rho&D_-&-Cp_1&D_+\\{}_\alpha\rho&-Cp_1&D_+&-D_-p_1\\\alphastarrho&-D_+p_1&D_-&-Cp_1\end{array}}\\[0.5\baselineskip]F_{\rho}^{{}_\alpha\rho{}_\alpha\rho{}_\alpha\rho}=&{\small\begin{array}{c|ccc}&\rho&{}_\alpha\rho&\alphastarrho\\\hline\rho&C&D_+&D_-\\{}_\alpha\rho&D_+&D_-&C\\\alphastarrho&D_-&C&D_+\end{array}}&F_{\rho}^{{}_\alpha\rho{}_\alpha\rho\alphastarrho}=&{\small\begin{array}{c|ccc}&\rho&{}_\alpha\rho&\alphastarrho\\\hline\rho&D_+&D_-&Cp_1p_2\\{}_\alpha\rho&D_-&C&D_+p_1p_2\\\alphastarrho&C&D_+&D_-p_1p_2\end{array}}\\[0.5\baselineskip]F_{\rho}^{{}_\alpha\rho\alphastarrho\rho}=&{\small\begin{array}{c|ccc}&\rho&{}_\alpha\rho&\alphastarrho\\\hline\rho&-Cp_1&D_+&D_-\\{}_\alpha\rho&D_+p_1&-D_-&-C\\\alphastarrho&D_-&-Cp_1&-D_+p_1\end{array}}&F_{\rho}^{{}_\alpha\rho\alphastarrho\alphastarrho}=&{\small\begin{array}{c|ccc}&\rho&{}_\alpha\rho&\alphastarrho\\\hline\rho&D_-p_1p_2&Cp_1p_2&D_+\\{}_\alpha\rho&C&D_+&D_-p_1p_2\\\alphastarrho&D_+&D_-&Cp_1p_2\end{array}}\\[0.5\baselineskip]F_{\rho}^{\alphastarrho\rho\rho}=&{\small\begin{array}{c|ccc}&\rho&{}_\alpha\rho&\alphastarrho\\\hline\rho&D_-&Cp_1&D_+\\{}_\alpha\rho&C&D_+p_1&D_-\\\alphastarrho&-D_+p_1p_2&-D_-p_2&-Cp_1p_2\end{array}}&F_{\rho}^{\alphastarrho\rho\alphastarrho}=&{\small\begin{array}{c|ccc}&\rho&{}_\alpha\rho&\alphastarrho\\\hline\rho&D_+&-D_-&-Cp_2\\{}_\alpha\rho&-D_-p_2&Cp_2&D_+\\\alphastarrho&-Cp_2&D_+p_2&D_-\end{array}}\\[0.5\baselineskip]F_{\rho}^{\alphastarrho{}_\alpha\rho\rho}=&{\small\begin{array}{c|ccc}&\rho&{}_\alpha\rho&\alphastarrho\\\hline\rho&-Cp_1&D_+&D_-\\{}_\alpha\rho&D_+&-D_-p_1&-Cp_1\\\alphastarrho&-D_-p_2&Cp_1p_2&D_+p_1p_2\end{array}}&F_{\rho}^{\alphastarrho{}_\alpha\rho{}_\alpha\rho}=&{\small\begin{array}{c|ccc}&\rho&{}_\alpha\rho&\alphastarrho\\\hline\rho&D_+&D_-&C\\{}_\alpha\rho&D_-&C&D_+\\\alphastarrho&Cp_1p_2&D_+p_1p_2&D_-p_1p_2\end{array}}\\[0.5\baselineskip]F_{\rho}^{\alphastarrho\alphastarrho{}_\alpha\rho}=&{\small\begin{array}{c|ccc}&\rho&{}_\alpha\rho&\alphastarrho\\\hline\rho&D_-p_1p_2&C&D_+\\{}_\alpha\rho&Cp_1p_2&D_+&D_-\\\alphastarrho&D_+&D_-p_1p_2&Cp_1p_2\end{array}}&F_{\rho}^{\alphastarrho\alphastarrho\alphastarrho}=&{\small\begin{array}{c|ccc}&\rho&{}_\alpha\rho&\alphastarrho\\\hline\rho&C&D_+p_1p_2&D_-\\{}_\alpha\rho&D_+p_1p_2&D_-&Cp_1p_2\\\alphastarrho&D_-&Cp_1p_2&D_+\end{array}}\\[0.5\baselineskip]F_{{}_\alpha\rho}^{\rho\rho\rho}=&{\small\begin{array}{c|ccc}&\rho&{}_\alpha\rho&\alphastarrho\\\hline\rho&D_+&D_-&-Cp_1\\{}_\alpha\rho&D_-&C&-D_+p_1\\\alphastarrho&-Cp_1p_2&-D_+p_1p_2&D_-p_2\end{array}}&F_{{}_\alpha\rho}^{\rho\rho\alphastarrho}=&{\small\begin{array}{c|ccc}&\rho&{}_\alpha\rho&\alphastarrho\\\hline\rho&-Cp_2&D_+&-D_-\\{}_\alpha\rho&D_+&-D_-p_2&Cp_2\\\alphastarrho&D_-&-Cp_2&D_+p_2\end{array}}\\[0.5\baselineskip]F_{{}_\alpha\rho}^{\rho{}_\alpha\rho\rho}=&{\small\begin{array}{c|ccc}&\rho&{}_\alpha\rho&\alphastarrho\\\hline\rho&D_-&-Cp_1&-D_+\\{}_\alpha\rho&-Cp_1&D_+&D_-p_1\\\alphastarrho&D_+p_1p_2&-D_-p_2&-Cp_1p_2\end{array}}&F_{{}_\alpha\rho}^{\rho{}_\alpha\rho{}_\alpha\rho}=&{\small\begin{array}{c|ccc}&\rho&{}_\alpha\rho&\alphastarrho\\\hline\rho&C&D_+&D_-\\{}_\alpha\rho&D_+&D_-&C\\\alphastarrho&D_-p_1p_2&Cp_1p_2&D_+p_1p_2\end{array}}\\[0.5\baselineskip]F_{{}_\alpha\rho}^{\rho\alphastarrho{}_\alpha\rho}=&{\small\begin{array}{c|ccc}&\rho&{}_\alpha\rho&\alphastarrho\\\hline\rho&D_+&D_-p_1p_2&C\\{}_\alpha\rho&D_-&Cp_1p_2&D_+\\\alphastarrho&Cp_1p_2&D_+&D_-p_1p_2\end{array}}&F_{{}_\alpha\rho}^{\rho\alphastarrho\alphastarrho}=&{\small\begin{array}{c|ccc}&\rho&{}_\alpha\rho&\alphastarrho\\\hline\rho&D_-&C&D_+p_1p_2\\{}_\alpha\rho&Cp_1p_2&D_+p_1p_2&D_-\\\alphastarrho&D_+&D_-&Cp_1p_2\end{array}}\\[0.5\baselineskip]F_{{}_\alpha\rho}^{{}_\alpha\rho\rho{}_\alpha\rho}=&{\small\begin{array}{c|ccc}&\rho&{}_\alpha\rho&\alphastarrho\\\hline\rho&C&D_+&-D_-p_1\\{}_\alpha\rho&D_+&D_-&-Cp_1\\\alphastarrho&D_-&C&-D_+p_1\end{array}}&F_{{}_\alpha\rho}^{{}_\alpha\rho\rho\alphastarrho}=&{\small\begin{array}{c|ccc}&\rho&{}_\alpha\rho&\alphastarrho\\\hline\rho&D_+p_1p_2&D_-&-Cp_1\\{}_\alpha\rho&D_-&Cp_1p_2&-D_+p_2\\\alphastarrho&Cp_1p_2&D_+&-D_-p_1\end{array}}\\[0.5\baselineskip]F_{{}_\alpha\rho}^{{}_\alpha\rho{}_\alpha\rho\rho}=&{\small\begin{array}{c|ccc}&\rho&{}_\alpha\rho&\alphastarrho\\\hline\rho&C&D_+&D_-p_1\\{}_\alpha\rho&D_+&D_-&Cp_1\\\alphastarrho&-D_-p_1&-Cp_1&-D_+\end{array}}&F_{{}_\alpha\rho}^{{}_\alpha\rho{}_\alpha\rho\alphastarrho}=&{\small\begin{array}{c|ccc}&\rho&{}_\alpha\rho&\alphastarrho\\\hline\rho&D_-p_1p_2&-Cp_1&-D_+p_1\\{}_\alpha\rho&-Cp_2&D_+&D_-\\\alphastarrho&-D_+p_2&D_-&C\end{array}}\\[0.5\baselineskip]F_{{}_\alpha\rho}^{{}_\alpha\rho\alphastarrho\rho}=&{\small\begin{array}{c|ccc}&\rho&{}_\alpha\rho&\alphastarrho\\\hline\rho&D_+&D_-&Cp_1\\{}_\alpha\rho&D_-p_1&Cp_1&D_+\\\alphastarrho&C&D_+&D_-p_1\end{array}}&F_{{}_\alpha\rho}^{{}_\alpha\rho\alphastarrho{}_\alpha\rho}=&{\small\begin{array}{c|ccc}&\rho&{}_\alpha\rho&\alphastarrho\\\hline\rho&D_-&-Cp_1&-D_+p_1\\{}_\alpha\rho&-Cp_1&D_+&D_-\\\alphastarrho&-D_+p_1&D_-&C\end{array}}\\[0.5\baselineskip]F_{{}_\alpha\rho}^{\alphastarrho\rho\rho}=&{\small\begin{array}{c|ccc}&\rho&{}_\alpha\rho&\alphastarrho\\\hline\rho&-C&D_+&D_-\\{}_\alpha\rho&D_+&-D_-&-C\\\alphastarrho&-D_-&C&D_+\end{array}}&F_{{}_\alpha\rho}^{\alphastarrho\rho{}_\alpha\rho}=&{\small\begin{array}{c|ccc}&\rho&{}_\alpha\rho&\alphastarrho\\\hline\rho&D_+p_1&D_-&Cp_1\\{}_\alpha\rho&D_-&Cp_1&D_+\\\alphastarrho&C&D_+p_1&D_-\end{array}}\\[0.5\baselineskip]F_{{}_\alpha\rho}^{\alphastarrho{}_\alpha\rho{}_\alpha\rho}=&{\small\begin{array}{c|ccc}&\rho&{}_\alpha\rho&\alphastarrho\\\hline\rho&D_-p_1&-C&-D_+\\{}_\alpha\rho&-Cp_1&D_+&D_-\\\alphastarrho&-D_+p_1&D_-&C\end{array}}&F_{{}_\alpha\rho}^{\alphastarrho{}_\alpha\rho\alphastarrho}=&{\small\begin{array}{c|ccc}&\rho&{}_\alpha\rho&\alphastarrho\\\hline\rho&Cp_2&-D_+&-D_-\\{}_\alpha\rho&-D_+p_2&D_-&C\\\alphastarrho&-D_-p_2&C&D_+\end{array}}\\[0.5\baselineskip]F_{{}_\alpha\rho}^{\alphastarrho\alphastarrho\rho}=&{\small\begin{array}{c|ccc}&\rho&{}_\alpha\rho&\alphastarrho\\\hline\rho&D_-p_2&Cp_1&D_+\\{}_\alpha\rho&Cp_2&D_+p_1&D_-\\\alphastarrho&D_+p_1p_2&D_-&Cp_1\end{array}}&F_{{}_\alpha\rho}^{\alphastarrho\alphastarrho\alphastarrho}=&{\small\begin{array}{c|ccc}&\rho&{}_\alpha\rho&\alphastarrho\\\hline\rho&D_+p_2&-D_-p_1p_2&-Cp_1p_2\\{}_\alpha\rho&-D_-p_1&C&D_+\\\alphastarrho&-Cp_1&D_+&D_-\end{array}}\\[0.5\baselineskip]F_{\alphastarrho}^{\rho\rho\rho}=&{\small\begin{array}{c|ccc}&\rho&{}_\alpha\rho&\alphastarrho\\\hline\rho&D_-&Cp_1&D_+\\{}_\alpha\rho&-C&-D_+p_1&-D_-\\\alphastarrho&D_+&D_-p_1&C\end{array}}&F_{\alphastarrho}^{\rho\rho{}_\alpha\rho}=&{\small\begin{array}{c|ccc}&\rho&{}_\alpha\rho&\alphastarrho\\\hline\rho&Cp_1&-D_+&D_-\\{}_\alpha\rho&D_+&-D_-p_1&Cp_1\\\alphastarrho&D_-&-Cp_1&D_+p_1\end{array}}\\[0.5\baselineskip]F_{\alphastarrho}^{\rho{}_\alpha\rho{}_\alpha\rho}=&{\small\begin{array}{c|ccc}&\rho&{}_\alpha\rho&\alphastarrho\\\hline\rho&-D_+&-D_-&-C\\{}_\alpha\rho&D_-&C&D_+\\\alphastarrho&C&D_+&D_-\end{array}}&F_{\alphastarrho}^{\rho{}_\alpha\rho\alphastarrho}=&{\small\begin{array}{c|ccc}&\rho&{}_\alpha\rho&\alphastarrho\\\hline\rho&-D_-&-C&-D_+\\{}_\alpha\rho&C&D_+&D_-\\\alphastarrho&D_+&D_-&C\end{array}}\\[0.5\baselineskip]F_{\alphastarrho}^{\rho\alphastarrho\rho}=&{\small\begin{array}{c|ccc}&\rho&{}_\alpha\rho&\alphastarrho\\\hline\rho&D_+&-D_-p_1p_2&Cp_1\\{}_\alpha\rho&D_-&-Cp_1p_2&D_+p_1\\\alphastarrho&Cp_1&-D_+p_2&D_-\end{array}}&F_{\alphastarrho}^{\rho\alphastarrho\alphastarrho}=&{\small\begin{array}{c|ccc}&\rho&{}_\alpha\rho&\alphastarrho\\\hline\rho&-Cp_1p_2&-D_+&-D_-p_1p_2\\{}_\alpha\rho&D_+&D_-p_1p_2&C\\\alphastarrho&D_-&Cp_1p_2&D_+\end{array}}\\[0.5\baselineskip]F_{\alphastarrho}^{{}_\alpha\rho\rho\rho}=&{\small\begin{array}{c|ccc}&\rho&{}_\alpha\rho&\alphastarrho\\\hline\rho&-C&D_+&D_-\\{}_\alpha\rho&-D_+&D_-&C\\\alphastarrho&D_-&-C&-D_+\end{array}}&F_{\alphastarrho}^{{}_\alpha\rho\rho\alphastarrho}=&{\small\begin{array}{c|ccc}&\rho&{}_\alpha\rho&\alphastarrho\\\hline\rho&-D_-p_1&-Cp_1&D_+\\{}_\alpha\rho&Cp_1p_2&D_+p_1p_2&-D_-p_2\\\alphastarrho&D_+&D_-&-Cp_1\end{array}}\\[0.5\baselineskip]F_{\alphastarrho}^{{}_\alpha\rho{}_\alpha\rho\rho}=&{\small\begin{array}{c|ccc}&\rho&{}_\alpha\rho&\alphastarrho\\\hline\rho&D_+p_1&D_-&-Cp_1\\{}_\alpha\rho&-D_-&-Cp_1&D_+\\\alphastarrho&Cp_1&D_+&-D_-p_1\end{array}}&F_{\alphastarrho}^{{}_\alpha\rho{}_\alpha\rho{}_\alpha\rho}=&{\small\begin{array}{c|ccc}&\rho&{}_\alpha\rho&\alphastarrho\\\hline\rho&-D_-p_1&C&D_+\\{}_\alpha\rho&-Cp_1&D_+&D_-\\\alphastarrho&-D_+p_1&D_-&C\end{array}}\\[0.5\baselineskip]F_{\alphastarrho}^{{}_\alpha\rho\alphastarrho{}_\alpha\rho}=&{\small\begin{array}{c|ccc}&\rho&{}_\alpha\rho&\alphastarrho\\\hline\rho&-Cp_1&D_+&D_-\\{}_\alpha\rho&-D_+p_1&D_-&C\\\alphastarrho&-D_-p_1&C&D_+\end{array}}&F_{\alphastarrho}^{{}_\alpha\rho\alphastarrho\alphastarrho}=&{\small\begin{array}{c|ccc}&\rho&{}_\alpha\rho&\alphastarrho\\\hline\rho&-D_+p_2&D_-p_1p_2&Cp_1p_2\\{}_\alpha\rho&-D_-p_1&C&D_+\\\alphastarrho&-Cp_1&D_+&D_-\end{array}}\\[0.5\baselineskip]F_{\alphastarrho}^{\alphastarrho\rho{}_\alpha\rho}=&{\small\begin{array}{c|ccc}&\rho&{}_\alpha\rho&\alphastarrho\\\hline\rho&D_-&Cp_1&D_+\\{}_\alpha\rho&C&D_+p_1&D_-\\\alphastarrho&D_+&D_-p_1&C\end{array}}&F_{\alphastarrho}^{\alphastarrho\rho\alphastarrho}=&{\small\begin{array}{c|ccc}&\rho&{}_\alpha\rho&\alphastarrho\\\hline\rho&C&D_+p_1&D_-\\{}_\alpha\rho&D_+p_1p_2&D_-p_2&Cp_1p_2\\\alphastarrho&D_-&Cp_1&D_+\end{array}}\\[0.5\baselineskip]F_{\alphastarrho}^{\alphastarrho{}_\alpha\rho\rho}=&{\small\begin{array}{c|ccc}&\rho&{}_\alpha\rho&\alphastarrho\\\hline\rho&-D_-&-Cp_1&D_+\\{}_\alpha\rho&-C&-D_+p_1&D_-\\\alphastarrho&D_+p_1&D_-&-Cp_1\end{array}}&F_{\alphastarrho}^{\alphastarrho{}_\alpha\rho\alphastarrho}=&{\small\begin{array}{c|ccc}&\rho&{}_\alpha\rho&\alphastarrho\\\hline\rho&D_+&-D_-p_1&-Cp_1\\{}_\alpha\rho&-D_-p_1&C&D_+\\\alphastarrho&-Cp_1&D_+&D_-\end{array}}\\[0.5\baselineskip]F_{\alphastarrho}^{\alphastarrho\alphastarrho\rho}=&{\small\begin{array}{c|ccc}&\rho&{}_\alpha\rho&\alphastarrho\\\hline\rho&-Cp_1p_2&-D_+p_1&D_-\\{}_\alpha\rho&-D_+p_2&-D_-&Cp_1\\\alphastarrho&-D_-p_1p_2&-Cp_1&D_+\end{array}}&F_{\alphastarrho}^{\alphastarrho\alphastarrho{}_\alpha\rho}=&{\small\begin{array}{c|ccc}&\rho&{}_\alpha\rho&\alphastarrho\\\hline\rho&D_+p_1p_2&-D_-p_1&-Cp_1\\{}_\alpha\rho&-D_-p_2&C&D_+\\\alphastarrho&-Cp_2&D_+&D_-\end{array}}
\end{align}}
\subsection{$4$-dimensional $F$-symbols}
{\allowdisplaybreaks\begin{align}
F_{\rho}^{\rho\rho\rho}=&{\small\begin{array}{c|cccc}&1&\rho&{}_\alpha\rho&\alphastarrho\\\hline1&A&\sqrt{A}&-p_1\sqrt{A}&p_1\sqrt{A}\\\rho&\sqrt{A}&-B&-D_+p_1&D_-p_1\\{}_\alpha\rho&-p_1\sqrt{A}&-D_+p_1&D_-&B\\\alphastarrho&p_1\sqrt{A}&D_-p_1&B&D_+\end{array}}\\[0.5\baselineskip]F_{\rho}^{\rho{}_\alpha\rho{}_\alpha\rho}=&{\small\begin{array}{c|cccc}&\alpha^*&\rho&{}_\alpha\rho&\alphastarrho\\\hline1&A&-p_1\sqrt{A}&\sqrt{A}&\sqrt{A}\\\rho&\sqrt{A}&-D_-p_1&-B&D_+\\{}_\alpha\rho&-p_1\sqrt{A}&-B&-D_+p_1&-D_-p_1\\\alphastarrho&-p_1\sqrt{A}&D_+&-D_-p_1&Bp_1\end{array}}\\[0.5\baselineskip]F_{\rho}^{\rho\alphastarrho\alphastarrho}=&{\small\begin{array}{c|cccc}&\alpha&\rho&{}_\alpha\rho&\alphastarrho\\\hline1&A&p_1\sqrt{A}&-p_1p_2\sqrt{A}&-p_1p_2\sqrt{A}\\\rho&-p_1p_2\sqrt{A}&-D_+p_2&D_-&-B\\{}_\alpha\rho&p_1\sqrt{A}&D_-&Bp_2&-D_+p_2\\\alphastarrho&p_1\sqrt{A}&-B&-D_+p_2&-D_-p_2\end{array}}\\[0.5\baselineskip]F_{\rho}^{{}_\alpha\rho\rho\alphastarrho}=&{\small\begin{array}{c|cccc}&\alpha&\rho&{}_\alpha\rho&\alphastarrho\\\hline\alpha&-A&p_1\sqrt{A}&-\sqrt{A}&p_2\sqrt{A}\\\rho&-p_1\sqrt{A}&-B&-D_+p_1&D_-p_1p_2\\{}_\alpha\rho&p_1p_2\sqrt{A}&-D_+p_2&D_-p_1p_2&Bp_1\\\alphastarrho&\sqrt{A}&-D_-p_1&-B&-D_+p_2\end{array}}\\[0.5\baselineskip]F_{\rho}^{{}_\alpha\rho{}_\alpha\rho\rho}=&{\small\begin{array}{c|cccc}&1&\rho&{}_\alpha\rho&\alphastarrho\\\hline\alpha&A&\sqrt{A}&-p_1\sqrt{A}&-p_1\sqrt{A}\\\rho&-p_1\sqrt{A}&-D_-p_1&-B&D_+\\{}_\alpha\rho&\sqrt{A}&-B&-D_+p_1&-D_-p_1\\\alphastarrho&-p_1\sqrt{A}&-D_+p_1&D_-&-B\end{array}}\\[0.5\baselineskip]F_{\rho}^{{}_\alpha\rho\alphastarrho{}_\alpha\rho}=&{\small\begin{array}{c|cccc}&\alpha^*&\rho&{}_\alpha\rho&\alphastarrho\\\hline\alpha&-A&-p_1\sqrt{A}&-p_1\sqrt{A}&-p_1\sqrt{A}\\\rho&p_1\sqrt{A}&D_+&D_-&-B\\{}_\alpha\rho&p_1\sqrt{A}&D_-&-B&D_+\\\alphastarrho&p_1\sqrt{A}&-B&D_+&D_-\end{array}}\\[0.5\baselineskip]F_{\rho}^{\alphastarrho\rho{}_\alpha\rho}=&{\small\begin{array}{c|cccc}&\alpha^*&\rho&{}_\alpha\rho&\alphastarrho\\\hline\alpha^*&-A&p_1\sqrt{A}&-p_1\sqrt{A}&-\sqrt{A}\\\rho&-p_1\sqrt{A}&-B&-D_+&-D_-p_1\\{}_\alpha\rho&\sqrt{A}&-D_+p_1&D_-p_1&-B\\\alphastarrho&p_1p_2\sqrt{A}&-D_-p_2&-Bp_2&D_+p_1p_2\end{array}}\\[0.5\baselineskip]F_{\rho}^{\alphastarrho{}_\alpha\rho\alphastarrho}=&{\small\begin{array}{c|cccc}&\alpha&\rho&{}_\alpha\rho&\alphastarrho\\\hline\alpha^*&A&-p_1\sqrt{A}&-p_1\sqrt{A}&-p_2\sqrt{A}\\\rho&-p_1\sqrt{A}&D_-&-B&D_+p_1p_2\\{}_\alpha\rho&-p_1\sqrt{A}&-B&D_+&D_-p_1p_2\\\alphastarrho&-p_2\sqrt{A}&D_+p_1p_2&D_-p_1p_2&-B\end{array}}\\[0.5\baselineskip]F_{\rho}^{\alphastarrho\alphastarrho\rho}=&{\small\begin{array}{c|cccc}&1&\rho&{}_\alpha\rho&\alphastarrho\\\hline\alpha^*&A&-p_1p_2\sqrt{A}&p_1\sqrt{A}&p_1\sqrt{A}\\\rho&p_1\sqrt{A}&-D_+p_2&D_-&-B\\{}_\alpha\rho&-p_1\sqrt{A}&D_-p_2&B&-D_+\\\alphastarrho&-p_1p_2\sqrt{A}&-B&-D_+p_2&-D_-p_2\end{array}}\\[0.5\baselineskip]F_{{}_\alpha\rho}^{\rho\rho{}_\alpha\rho}=&{\small\begin{array}{c|cccc}&1&\rho&{}_\alpha\rho&\alphastarrho\\\hline\alpha^*&A&\sqrt{A}&-p_1\sqrt{A}&p_1\sqrt{A}\\\rho&-p_1\sqrt{A}&-D_-p_1&-B&-D_+\\{}_\alpha\rho&\sqrt{A}&-B&-D_+p_1&D_-p_1\\\alphastarrho&p_1p_2\sqrt{A}&D_+p_1p_2&-D_-p_2&-Bp_2\end{array}}\\[0.5\baselineskip]F_{{}_\alpha\rho}^{\rho{}_\alpha\rho\alphastarrho}=&{\small\begin{array}{c|cccc}&\alpha^*&\rho&{}_\alpha\rho&\alphastarrho\\\hline\alpha^*&Ap_1&p_2\sqrt{A}&p_1\sqrt{A}&p_1\sqrt{A}\\\rho&\sqrt{A}&D_+p_1p_2&D_-&-B\\{}_\alpha\rho&\sqrt{A}&D_-p_1p_2&-B&D_+\\\alphastarrho&p_1p_2\sqrt{A}&-B&D_+p_1p_2&D_-p_1p_2\end{array}}\\[0.5\baselineskip]F_{{}_\alpha\rho}^{\rho\alphastarrho\rho}=&{\small\begin{array}{c|cccc}&\alpha&\rho&{}_\alpha\rho&\alphastarrho\\\hline\alpha^*&-A&-p_1\sqrt{A}&p_1p_2\sqrt{A}&p_1\sqrt{A}\\\rho&p_1\sqrt{A}&-B&-D_+p_2&-D_-\\{}_\alpha\rho&-p_1\sqrt{A}&-D_+&D_-p_2&-B\\\alphastarrho&-p_1p_2\sqrt{A}&-D_-p_2&-B&D_+p_2\end{array}}\\[0.5\baselineskip]F_{{}_\alpha\rho}^{{}_\alpha\rho\rho\rho}=&{\small\begin{array}{c|cccc}&\alpha&\rho&{}_\alpha\rho&\alphastarrho\\\hline1&A&-p_1\sqrt{A}&\sqrt{A}&-\sqrt{A}\\\rho&\sqrt{A}&-D_-p_1&-B&-D_+\\{}_\alpha\rho&-p_1\sqrt{A}&-B&-D_+p_1&D_-p_1\\\alphastarrho&p_1\sqrt{A}&-D_+&D_-p_1&Bp_1\end{array}}\\[0.5\baselineskip]F_{{}_\alpha\rho}^{{}_\alpha\rho{}_\alpha\rho{}_\alpha\rho}=&{\small\begin{array}{c|cccc}&1&\rho&{}_\alpha\rho&\alphastarrho\\\hline1&A&\sqrt{A}&-p_1\sqrt{A}&-p_1\sqrt{A}\\\rho&\sqrt{A}&D_+&-D_-p_1&Bp_1\\{}_\alpha\rho&-p_1\sqrt{A}&-D_-p_1&-B&D_+\\\alphastarrho&-p_1\sqrt{A}&Bp_1&D_+&D_-\end{array}}\\[0.5\baselineskip]F_{{}_\alpha\rho}^{{}_\alpha\rho\alphastarrho\alphastarrho}=&{\small\begin{array}{c|cccc}&\alpha^*&\rho&{}_\alpha\rho&\alphastarrho\\\hline1&A&-p_1p_2\sqrt{A}&p_1\sqrt{A}&p_1\sqrt{A}\\\rho&-p_1p_2\sqrt{A}&-B&-D_+p_2&-D_-p_2\\{}_\alpha\rho&p_1\sqrt{A}&-D_+p_2&D_-&-B\\\alphastarrho&p_1\sqrt{A}&-D_-p_2&-B&D_+\end{array}}\\[0.5\baselineskip]F_{{}_\alpha\rho}^{\alphastarrho\rho\alphastarrho}=&{\small\begin{array}{c|cccc}&\alpha^*&\rho&{}_\alpha\rho&\alphastarrho\\\hline\alpha&A&p_1p_2\sqrt{A}&p_1\sqrt{A}&\sqrt{A}\\\rho&p_1\sqrt{A}&D_-p_2&-B&D_+p_1\\{}_\alpha\rho&p_1p_2\sqrt{A}&-B&D_+p_2&D_-p_1p_2\\\alphastarrho&\sqrt{A}&D_+p_1p_2&D_-p_1&-B\end{array}}\\[0.5\baselineskip]F_{{}_\alpha\rho}^{\alphastarrho{}_\alpha\rho\rho}=&{\small\begin{array}{c|cccc}&\alpha&\rho&{}_\alpha\rho&\alphastarrho\\\hline\alpha&Ap_1&-\sqrt{A}&-\sqrt{A}&-p_1\sqrt{A}\\\rho&-\sqrt{A}&D_+p_1&D_-p_1&-B\\{}_\alpha\rho&-p_1\sqrt{A}&D_-&-B&D_+p_1\\\alphastarrho&\sqrt{A}&Bp_1&-D_+p_1&-D_-\end{array}}\\[0.5\baselineskip]F_{{}_\alpha\rho}^{\alphastarrho\alphastarrho{}_\alpha\rho}=&{\small\begin{array}{c|cccc}&1&\rho&{}_\alpha\rho&\alphastarrho\\\hline\alpha&A&-p_1p_2\sqrt{A}&p_1\sqrt{A}&p_1\sqrt{A}\\\rho&-p_1\sqrt{A}&-Bp_2&-D_+&-D_-\\{}_\alpha\rho&p_1\sqrt{A}&-D_+p_2&D_-&-B\\\alphastarrho&p_1\sqrt{A}&-D_-p_2&-B&D_+\end{array}}\\[0.5\baselineskip]F_{\alphastarrho}^{\rho\rho\alphastarrho}=&{\small\begin{array}{c|cccc}&1&\rho&{}_\alpha\rho&\alphastarrho\\\hline\alpha&A&\sqrt{A}&-p_1\sqrt{A}&p_1\sqrt{A}\\\rho&p_1\sqrt{A}&D_+p_1&-D_-&-B\\{}_\alpha\rho&p_1p_2\sqrt{A}&D_-p_1p_2&Bp_2&D_+p_2\\\alphastarrho&\sqrt{A}&-B&-D_+p_1&D_-p_1\end{array}}\\[0.5\baselineskip]F_{\alphastarrho}^{\rho{}_\alpha\rho\rho}=&{\small\begin{array}{c|cccc}&\alpha^*&\rho&{}_\alpha\rho&\alphastarrho\\\hline\alpha&-A&-p_1\sqrt{A}&p_1\sqrt{A}&-\sqrt{A}\\\rho&p_1\sqrt{A}&-B&-D_+&D_-p_1\\{}_\alpha\rho&\sqrt{A}&D_+p_1&-D_-p_1&-B\\\alphastarrho&-p_1\sqrt{A}&-D_-&-B&-D_+p_1\end{array}}\\[0.5\baselineskip]F_{\alphastarrho}^{\rho\alphastarrho{}_\alpha\rho}=&{\small\begin{array}{c|cccc}&\alpha&\rho&{}_\alpha\rho&\alphastarrho\\\hline\alpha&-Ap_1&p_1\sqrt{A}&p_2\sqrt{A}&p_1\sqrt{A}\\\rho&\sqrt{A}&-D_-&Bp_1p_2&-D_+\\{}_\alpha\rho&-\sqrt{A}&-B&D_+p_1p_2&D_-\\\alphastarrho&-\sqrt{A}&D_+&D_-p_1p_2&-B\end{array}}\\[0.5\baselineskip]F_{\alphastarrho}^{{}_\alpha\rho\rho{}_\alpha\rho}=&{\small\begin{array}{c|cccc}&\alpha&\rho&{}_\alpha\rho&\alphastarrho\\\hline\alpha^*&-A&-\sqrt{A}&-\sqrt{A}&p_1\sqrt{A}\\\rho&-p_1\sqrt{A}&-D_+p_1&-D_-p_1&-B\\{}_\alpha\rho&\sqrt{A}&D_-&-B&-D_+p_1\\\alphastarrho&\sqrt{A}&-B&D_+&-D_-p_1\end{array}}\\[0.5\baselineskip]F_{\alphastarrho}^{{}_\alpha\rho{}_\alpha\rho\alphastarrho}=&{\small\begin{array}{c|cccc}&1&\rho&{}_\alpha\rho&\alphastarrho\\\hline\alpha^*&A&\sqrt{A}&-p_1\sqrt{A}&-p_1\sqrt{A}\\\rho&-p_1\sqrt{A}&Bp_1&D_+&D_-\\{}_\alpha\rho&-p_1\sqrt{A}&-D_+p_1&D_-&-B\\\alphastarrho&-p_1\sqrt{A}&-D_-p_1&-B&D_+\end{array}}\\[0.5\baselineskip]F_{\alphastarrho}^{{}_\alpha\rho\alphastarrho\rho}=&{\small\begin{array}{c|cccc}&\alpha^*&\rho&{}_\alpha\rho&\alphastarrho\\\hline\alpha^*&-Ap_1&\sqrt{A}&p_1\sqrt{A}&-\sqrt{A}\\\rho&-\sqrt{A}&D_-p_1&-B&-D_+p_1\\{}_\alpha\rho&\sqrt{A}&Bp_1&-D_+&D_-p_1\\\alphastarrho&p_1\sqrt{A}&-D_+&-D_-p_1&-B\end{array}}\\[0.5\baselineskip]F_{\alphastarrho}^{\alphastarrho\rho\rho}=&{\small\begin{array}{c|cccc}&\alpha^*&\rho&{}_\alpha\rho&\alphastarrho\\\hline1&A&p_1\sqrt{A}&-\sqrt{A}&\sqrt{A}\\\rho&\sqrt{A}&D_+p_1&-D_-&-B\\{}_\alpha\rho&-p_1\sqrt{A}&-D_-&-Bp_1&-D_+p_1\\\alphastarrho&p_1\sqrt{A}&-B&-D_+p_1&D_-p_1\end{array}}\\[0.5\baselineskip]F_{\alphastarrho}^{\alphastarrho{}_\alpha\rho{}_\alpha\rho}=&{\small\begin{array}{c|cccc}&\alpha&\rho&{}_\alpha\rho&\alphastarrho\\\hline1&A&\sqrt{A}&-p_1\sqrt{A}&-p_1\sqrt{A}\\\rho&\sqrt{A}&-B&-D_+p_1&-D_-p_1\\{}_\alpha\rho&-p_1\sqrt{A}&-D_+p_1&D_-&-B\\\alphastarrho&-p_1\sqrt{A}&-D_-p_1&-B&D_+\end{array}}\\[0.5\baselineskip]F_{\alphastarrho}^{\alphastarrho\alphastarrho\alphastarrho}=&{\small\begin{array}{c|cccc}&1&\rho&{}_\alpha\rho&\alphastarrho\\\hline1&A&-p_1p_2\sqrt{A}&p_1\sqrt{A}&p_1\sqrt{A}\\\rho&-p_1p_2\sqrt{A}&D_-&Bp_2&-D_+p_2\\{}_\alpha\rho&p_1\sqrt{A}&Bp_2&D_+&D_-\\\alphastarrho&p_1\sqrt{A}&-D_+p_2&D_-&-B\end{array}}
\end{align}}

\end{document}